\documentclass[11pt]{article}

\usepackage[utf8]{inputenc} 
\usepackage[T1]{fontenc}    
\usepackage{booktabs}       
\usepackage{nicefrac}       
\usepackage{microtype}      

\usepackage{amsmath,amsthm,amssymb,amsfonts,bbm}
\usepackage{geometry}
\usepackage{xr-hyper}
\usepackage{times,enumitem}
\usepackage[colorlinks,citecolor=blue,urlcolor=blue,linkcolor=blue]{hyperref}
\usepackage[colorlinks]{hyperref}
\usepackage{verbatim,float,url}
\usepackage{graphicx,psfrag}
\usepackage{subcaption}
\usepackage{algorithm,algorithmic}
\usepackage{natbib}
\usepackage{dsfont}

\newtheorem{theorem}{Theorem}
\newtheorem{lemma}[theorem]{Lemma}

\theoremstyle{remark}
\newtheorem{remark}{Remark}
\theoremstyle{definition}

\newcommand{\argmin}{\mathop{\mathrm{argmin}}}

\def\E{\mathbb{E}}
\def\P{\mathbb{P}}

\def\col{\mathrm{col}}
\def\row{\mathrm{row}}

\def\diag{\mathrm{diag}}

\def\hbeta{\hat{\beta}}

\def\htheta{\hat{\theta}}
\def\halpha{\hat{\alpha}}

\def\R{\mathbb{R}}

\def\cE{\mathcal{E}}

\def\cH{\mathcal{H}}

\def\cK{\mathcal{K}}

\def\cS{\mathcal{S}}
\def\cT{\mathcal{T}}

\def\1{\mathds{1}}

\def\MSE{\mathrm{MSE}}
\def\BR{C} 

\graphicspath{{../plots/}}

\begin{document}

\title{Total Variation Classes Beyond 1d: Minimax Rates, and the
  Limitations of Linear Smoothers}

\author{Veeranjaneyulu Sadhanala$^*$ \and
Yu-Xiang Wang$^*$ \and  Ryan J. Tibshirani}

\date{Carnegie Mellon University \\
{(\it $^*$These authors contributed equally})}

\maketitle

\begin{abstract}
We consider the problem of estimating a function defined over
$n$ locations on a $d$-dimensional grid (having all side lengths equal
to \smash{$n^{1/d}$}).  When the function is constrained to
have discrete total variation bounded by $\BR_n$, we derive the
minimax optimal (squared) $\ell_2$ estimation error rate, parametrized
by $n$ and $\BR_n$. Total variation denoising, also known as the fused
lasso, is seen to be rate optimal.  Several simpler estimators exist,
such as Laplacian smoothing and Laplacian eigenmaps.  A natural
question is: can these simpler estimators perform just as well?  We
prove that
these estimators, and more broadly all estimators given by linear
transformations of the input data, are suboptimal over the class of
functions with bounded variation. This extends fundamental findings
of \citet{minimaxwave} on 1-dimensional total variation spaces to
higher dimensions.  The implication is that the computationally
simpler methods cannot be used for such sophisticated denoising tasks,
without sacrificing statistical accuracy.
We also derive minimax rates for discrete Sobolev
spaces over $d$-dimensional grids, which are, in some sense,
smaller than the total variation function spaces.  Indeed, these are
small enough spaces that linear estimators can be optimal---and a few
well-known ones are, such as Laplacian smoothing and Laplacian
eigenmaps, as we show.  Lastly, we investigate the problem of
adaptivity of the total variation denoiser to these smaller Sobolev
function spaces.
\end{abstract}

\section{Introduction}

Let $G=(V,E)$ be a $d$-dimensional grid graph, i.e., a lattice graph,
with equal side lengths.  Label the nodes as $V=\{1,\ldots,n\}$, and
edges as $E=\{e_1,\ldots,e_m\}$. Consider data $y=(y_1,\ldots,y_n) \in
\R^n$ observed over the nodes, from a model
\begin{equation}
\label{eq:model}
y_i \sim N(\theta_{0,i}, \sigma^2), \quad
\text{i.i.d., for $i=1,\ldots,n$},
\end{equation}
where $\theta_0=(\theta_{0,1},\ldots,\theta_{0,n})\in \R^n$ is an
unknown mean parameter to be estimated, and $\sigma^2>0$ is the
marginal noise variance.  It is assumed that $\theta_0$ displays some
kind of regularity over the grid $G$, e.g., \smash{$\theta_0 \in
  \cT_d(\BR_n)$} for some $\BR_n>0$, where
\begin{equation}
\label{eq:tv_class}
\cT_d(\BR_n) = \big\{ \theta : \|D\theta\|_1 \leq \BR_n \big\},
\end{equation}
and $D \in \R^{m \times n}$ is the edge incidence matrix of $G$.  This
has $\ell$th row $D_\ell=(0,\ldots,-1,\ldots,1,\ldots,0)$, with a
$-1$ in the $i$th location, and $1$ in the $j$th location, provided
that the $\ell$th edge is $e_\ell=(i,j)$ with $i<j$.  Equivalently,
$L=D^T D$ is the graph Laplacian matrix of $G$, and thus
\begin{equation*}
\|D \theta\|_1 = \sum_{(i,j) \in E} |\theta_i-\theta_j|,
\qquad \text{and} \qquad
\|D\theta\|_2^2 = \theta^T L \theta =
\sum_{(i,j) \in E} (\theta_i-\theta_j)^2.
\end{equation*}
We will refer to the class in \eqref{eq:tv_class} as a {\it discrete
total variation (TV) class}, and to the quantity
$\|D\theta_0\|_1$ as the discrete total variation of $\theta_0$,
though for simplicity we will often drop the word ``discrete''.

The problem of estimating $\theta_0$ given a total variation
bound as in \eqref{eq:tv_class} is of great importance in both
nonparametric statistics and signal processing, and has many
applications, e.g., changepoint detection for 1d grids, and image
denoising for 2d and 3d
grids.  There has been much methodological and computational work
devoted to this problem, resulting in practically efficient estimators
in dimensions 1, 2, 3, and beyond.  However, theoretical performance,
and in particularly optimality, is only really well-understood in the
1-dimensional setting.  This paper seeks to change that, and offers
theory in $d$-dimensions that parallel more classical results known in
the 1-dimensional case.

\paragraph{Estimators under consideration.}

Central role to our work is the
{\it total variation (TV) denoising} or {\it fused lasso} estimator
(e.g.,
\cite{tv,vogel1996iterative,
  fuse,fuseflow,hoefling2010path,
  genlasso,sharpnack2012sparsistency,
  barbero2014modular}),
defined by the convex optimization problem
\begin{equation}
\label{eq:tv}
\htheta{^\mathrm{TV}} = \argmin_{\theta \in \R^n} \;
\|y-\theta\|_2^2 + \lambda \|D\theta\|_1,
\end{equation}
where $\lambda \geq 0$ is a tuning parameter.  Another pair of
methods that we study carefully are {\it Laplacian smoothing}
and {\it Laplacian eigenmaps}, which are most commonly seen in the
context of clustering, dimensionality reduction, and semi-supervised
learning, but are also useful tools for estimation in a regression
setting like ours (e.g.,
\cite{belkin2002,belkin2003,smola2003,zhu2003,
  belkin2004,zhou2005,belkin2005,
  belkin2005towards,ando2006,
  sharpnack2010}).
The Laplacian smoothing estimator is given by
\begin{equation}
\label{eq:ls}
\htheta{^\mathrm{LS}} = \argmin_{\theta \in \R^n} \;
\|y-\theta\|_2^2 + \lambda \|D\theta\|_2^2,
\quad \text{i.e.,} \quad
\htheta{^\mathrm{LS}} = (I + \lambda L)^{-1} y,
\end{equation}
for a tuning parameter $\lambda \geq 0$, where in the second expression
we have written \smash{$\htheta^{\mathrm{LS}}$} in closed-form
(this is possible since it is the minimizer of a convex quadratic).
For Laplacian eigenmaps, we must introduce the
eigendecomposition of the graph Laplacian, $L=V \Sigma V^T$, where
$\Sigma=\diag(\rho_1,\ldots,\rho_n)$ with $0=\rho_1<
\rho_2 \leq \ldots \leq \rho_n$, and where
$V=[V_1,V_2,\ldots,V_n] \in \R^{n\times n}$ has orthonormal columns.
The Laplacian eigenmaps estimator is
\begin{equation}
\label{eq:le}
\htheta^{\mathrm{LE}} = V_{[k]} V_{[k]}^T y,
\quad \text{where} \quad
V_{[k]} = [V_1,V_2,\ldots,V_k] \in \R^{n\times k},
\end{equation}
where now $k \in \{1,\ldots,n\}$ acts as a tuning parameter.

Laplacian smoothing and Laplacian eigenmaps are appealing because
they are (relatively) simple: they are just linear transformations of
the data $y$.
Indeed, as we are considering $G$ to be a grid, both estimators in
\eqref{eq:ls}, \eqref{eq:le} can be computed
very quickly, in nearly $O(n)$ time, since the columns of $V$ here
are discrete cosine transform (DCT) basis vectors when
$d=1$, or Kronecker products thereof, when $d \geq 2$ (e.g.,
\cite{conte1980,godunov1987,
	kunsch1994,ng1999,wang2008}).
The TV denoising estimator in \eqref{eq:tv}, on the other hand, cannot
be expressed in closed-form, and is much more difficult to compute,
especially when $d \geq 2$, though several advances have been made
over the years (see the references above, and in particular
\cite{barbero2014modular} for an efficient operator-splitting
algorithm and nice literature survey). Importantly, these
computational difficulties are often worth it: TV denoising often
practically outperforms $\ell_2$-regularized estimators like Laplacian
smoothing (and also Laplacian eigenmaps) in image denoising tasks, as
it is able to better preserve sharp edges and object boundaries
(this is now widely accepted, early references are, e.g.,
\cite{acar1994,dobson1996,chambolle1997}).
See Figure \ref{fig:motivation} for an example, using the
often-studied ``cameraman'' image.

\begin{figure}[tb]
  \centering
  \begin{tabular}{ccc}
    Noisy image & Laplacian smoothing & TV denoising \\
    {\includegraphics[width=0.3\textwidth]{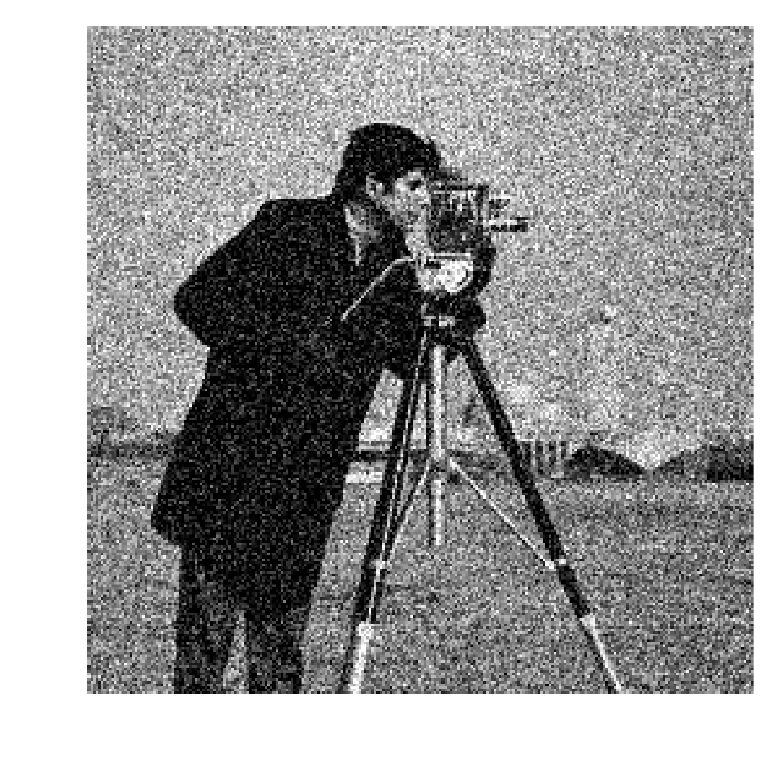}} &
    {\includegraphics[width=0.3\textwidth]{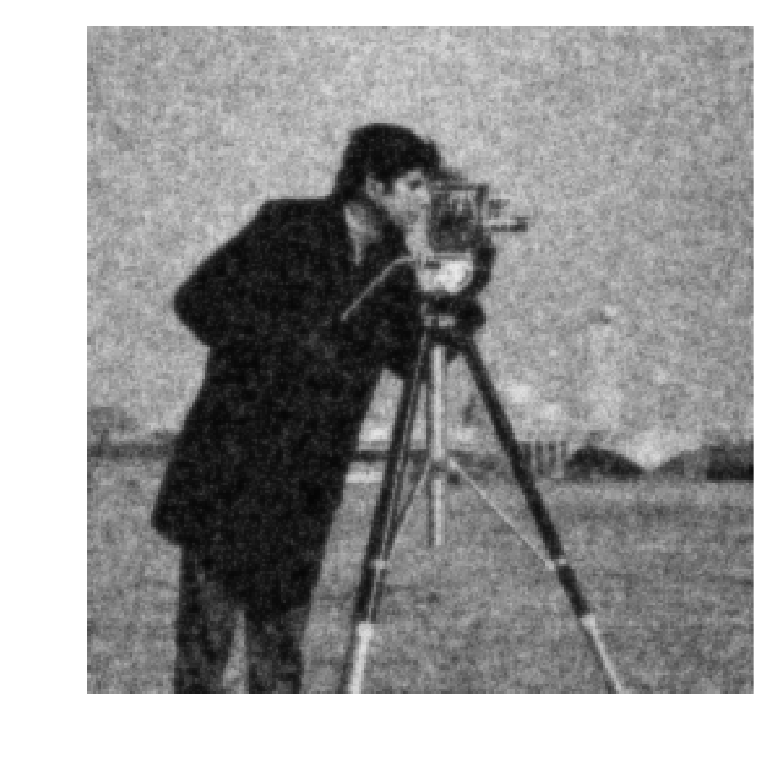}} &
    {\includegraphics[width=0.3\textwidth]{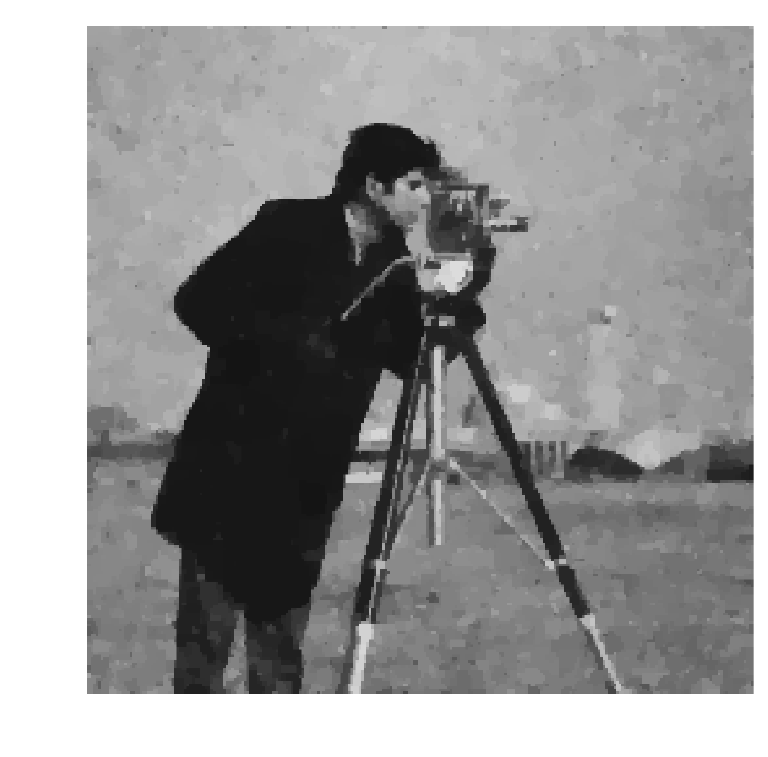}}
  \end{tabular}
  \vspace{-10pt}	
  \caption{\small\it Comparison of Laplacian smoothing and TV
    denoising for the common``cameraman'' image. TV denoising provides
    a more visually appealing result, and also achieves about a 35\%
    reduction in MSE compared to Laplacian smoothing (MSE being
    measured to the original image).  Both methods were tuned
    optimally.}
  \label{fig:motivation}
  \vspace{-8pt}	
\end{figure}

In the 1d setting, classical theory from nonparametric statistics
draws a clear distinction between the performance of TV denoising and
estimators like Laplacian smoothing and Laplacian eigenmaps.
Perhaps surprisingly, this theory has not yet been fully developed in
dimensions $d \geq 2$.  Arguably, the comparison between TV
denoising and Laplacian smoothing and Laplacian eigenmaps is even
more interesting in higher dimensions, because the computational gap
between the methods is even larger (the former method being much more
expensive, say in 2d and 3d, than the latter two).  Shortly, we review
the 1d theory, and what is known in $d$-dimensions, for $d \geq 2$.
First, we introduce notation.

\paragraph{Notation.}

For deterministic (nonrandom) sequences $a_n,b_n$ we write $a_n =
O(b_n)$ to denote that $a_n/b_n$ is upper bounded for all $n$ large
enough, and $a_n \asymp b_n$ to denote that both $a_n=O(b_n)$ and
$a^{-1}_n=O(b_n^{-1})$.   Also, for random sequences $A_n,B_n$, we
write $A_n = O_\P(B_n)$ to denote that $A_n/B_n$ is bounded in
probability.  We abbreviate $a \wedge b = \min\{a,b\}$ and
$a \vee b = \max\{a,b\}$.  For an estimator \smash{$\htheta$} of the
parameter $\theta_0$ in \eqref{eq:model}, we define its mean squared
error (MSE) to be
\begin{equation*}
\MSE(\htheta,\theta_0) = \frac{1}{n} \|\htheta-\theta_0\|_2^2.
\end{equation*}
The risk of \smash{$\htheta$} is the expectation of its MSE, and for
a set $\cK \subseteq \R^n$, we define the minimax risk
and minimax linear risk to be
\begin{equation*}
R(\cK) =
\inf_{\htheta} \; \sup_{\theta_0 \in \cK} \;
\E \big[\MSE(\htheta,\theta_0)\big]
\quad \text{and} \quad
R_L(\cK) =
\inf_{\htheta \, \text{linear}} \; \sup_{\theta_0 \in \cK} \;
\E \big[\MSE(\htheta,\theta_0)\big],
\end{equation*}
respectively,
where the infimum on in the first expression is over all estimators
\smash{$\htheta$}, and in the second expression over all
{\it linear estimators} \smash{$\htheta$}, meaning that
\smash{$\htheta = Sy$} for a matrix $S \in \R^{n\times n}$.  We will
also refer to linear estimators as {\it linear smoothers}. Note that
both Laplacian smoothing in \eqref{eq:ls} and Laplacian eigenmaps in
\eqref{eq:le} are linear smoothers, but TV denoising in \eqref{eq:tv} is
not. Lastly, in
somewhat of an abuse of nomenclature, we will often call the parameter
$\theta_0$ in \eqref{eq:model} a function, and a set of possible
values for $\theta_0$ as in \eqref{eq:tv_class} a function space; this
comes from thinking of the components of $\theta_0$ as the
evaluations of an underlying function over $n$ locations on
the grid. This embedding has no formal importance, but it is
convenient notationally, and matches the notation in nonparametric
statistics.

\paragraph{Review: TV denoising in 1d.}

The classical nonparametric statistics literature
\citep{donoho1990minimax,minimaxwave,locadapt} provides a more or less
complete story for estimation under total variation constraints in 1d.
See also \cite{trendfilter} for a translation of these results to a
setting more consistent (notationally) to that in the current
paper. Assume that $d=1$ and $\BR_n = C > 0$, a constant (not
growing with $n$).  The results in \cite{minimaxwave} imply that
\begin{equation}
\label{eq:minimax_1d}
R (\cT_1(C)) \asymp n^{-2/3}.
\end{equation}
Further, \cite{locadapt} showed that the TV denoiser
\smash{$\htheta^{\mathrm{TV}}$} in \eqref{eq:tv},
with \smash{$\lambda \asymp n^{1/3}$}, satisfies
\begin{equation}
\label{eq:tv_1d}
\MSE(\htheta^{\mathrm{TV}},\theta_0) = O_\P(n^{-2/3}),
\end{equation}
for all \smash{$\theta_0 \in \cT_1(C)$}, and is thus minimax rate
optimal over \smash{$\cT_1(C)$}. (In assessing rates here and
throughout, we do not distinguish between convergence in expectation
versus convergence in probability.)
Wavelet denoising, under various choices of wavelet
bases, also achieves the minimax rate. However, many simpler
estimators do not.  To be more precise, it is shown in
\cite{minimaxwave} that
\begin{equation}
\label{eq:minimax_1d_linear}
R_L (\cT_1(C)) \asymp n^{-1/2}.
\end{equation}
Therefore, a substantial number of commonly used nonparametric
estimators---such as running mean estimators, smoothing splines,
kernel smoothing, Laplacian smoothing, and Laplacian eigenmaps, which
are all linear smoothers---have a major deficiency when it comes to
estimating functions of bounded variation.  Roughly speaking, they
will require many more samples to estimate $\theta_0$ within the same
degree of accuracy as an optimal method like TV or wavelet denoising
(on the order of \smash{$\epsilon^{-1/2}$} times more samples to
achieve an MSE of $\epsilon$). Further theory and empirical examples
(e.g., \cite{donoho1994ideal,minimaxwave,trendfilter}) offer the
following perspective: linear smoothers cannot cope with functions in
$T(C)$ that have spatially inhomogeneous smoothness, i.e., that vary
smoothly at some locations and vary wildly at others.
Linear smoothers can only produce estimates that are smooth
throughout, or wiggly throughout, but not a mix of the two.  They can
hence perform well over smaller, more homogeneous function classes
like Sobolev or Holder classes, but not larger ones like total
variation classes (or more generally, Besov and Triebel classes), and
for these, one must use more sophisticated, nonlinear techniques. A
motivating question: does such a gap persist in higher
dimensions, between optimal nonlinear and linear estimators, and if
so, how big is it?

\paragraph{Review: TV denoising in multiple dimensions.}

Recently, \cite{graphtf} established rates for TV denoising over
various graph models, including grids, and \cite{hutter2016optimal}
made improvements, particularly in the case of $d$-dimensional
grids with $d \geq 2$.  We can combine Propositions 4 and 6 of
\cite{hutter2016optimal} with Theorem 3 of \cite{graphtf} to give the
following result: if $d \geq 2$, and $\BR_n$ is an arbitrary
sequence (potentially unbounded with $n$), then the TV denoiser
\smash{$\htheta^{\mathrm{TV}}$} in \eqref{eq:tv} satisfies, over all
\smash{$\theta_0 \in \cT_d(\BR_n)$},
\begin{equation}
\label{eq:tv_hd}
\MSE(\htheta^{\mathrm{TV}},\theta_0) =
O_\P\bigg( \frac{\BR_n \log{n}}{n} \bigg)
\;\, \text{for $d=2$}, \;\, \text{and} \;\,
\MSE(\htheta^{\mathrm{TV}},\theta_0) =
O_\P\bigg( \frac{\BR_n \sqrt{\log{n}}}{n} \bigg) \;\,
\text{for $d \geq 3$},
\end{equation}
with \smash{$\lambda \asymp \log{n}$} for $d=2$,
and \smash{$\lambda \asymp \sqrt{\log{n}}$} for $d \geq 3$.
Note that, at first glance, this is a very different result from the
1d case.  We expand on this next.

\section{Summary of results}

\paragraph{A gap in multiple dimensions.}

For estimation of $\theta_0$ in \eqref{eq:model} when $d \geq 2$,
consider, e.g., the simplest possible linear smoother: the
mean estimator, \smash{$\htheta^{\mathrm{mean}}=\bar{y} \1$} (where
$\1=(1,\ldots,1) \in \R^n$, the vector of all 1s).  Lemma
\ref{lem:mean}, given below, implies that over \smash{$\theta_0 \in
  \cT_d(\BR_n)$}, the MSE of the mean estimator is bounded in
probability by \smash{$\BR_n^2\log{n}/n$} for $d=2$, and
\smash{$\BR_n^2/n$} for $d \geq 3$.  Compare this to
\eqref{eq:tv_hd}. When
$\BR_n=C>0$ is a constant, i.e., when the TV of $\theta_0$ is assumed to
be bounded (which is assumed for the 1d results in
\eqref{eq:minimax_1d}, \eqref{eq:tv_1d},
\eqref{eq:minimax_1d_linear}), this means that the TV denoiser and the
mean estimator converge to $\theta_0$ {\it at the same rate},
basically (ignoring log terms), the
``parametric rate'' of $1/n$, for estimating a finite-dimensional
parameter!  That TV denoising and such a trivial
linear smoother perform comparably over 2d and 3d grids could not be
farther from the story in 1d, where TV denoising is separated by an
unbridgeable gap from {\it all} linear smoothers, as shown in
\eqref{eq:minimax_1d}, \eqref{eq:tv_1d}, \eqref{eq:minimax_1d_linear}.

Our results in Section \ref{sec:tv_class} clarify this conundrum,
and can be summarized by three points.

\begin{itemize}
\item We argue in Section \ref{sec:scaling} that there is
  a proper ``canonical'' scaling for the TV class defined in
  \eqref{eq:tv_class}.
  E.g., when $d=1$, this yields $\BR_n \asymp 1$, a constant,
  but when $d=2$, this yields \smash{$\BR_n \asymp \sqrt{n}$}, and
  $\BR_n$ also diverges with $n$ for all $d \geq 3$.
  Sticking with $d=2$ as an interesting example, we see that under
  such a scaling, the MSE rates achieved by TV denoising and the mean
  estimator respectively, are drastically different; ignoring log
  terms, these are
  \begin{equation}
    \label{eq:gap_2d}
    \frac{\BR_n}{n} \asymp \frac{1}{\sqrt{n}}
    \quad \text{and} \quad
    \frac{\BR_n^2}{n} \asymp 1,
  \end{equation}
  respectively.  Hence, TV denoising has an MSE rate of
  \smash{$1/\sqrt{n}$}, in a setting where the mean estimator has a
  {\it constant} rate, i.e., a setting where it is not even known to
  be consistent.
	
\item We show in Section \ref{sec:tv_minimax_linear} that our choice
  to study the mean estimator here is not somehow ``unlucky'' (it is
  not a particularly bad linear smoother, nor is the upper bound on
  its MSE loose):
  the minimax linear risk over \smash{$\cT_d(\BR_n)$} is on the order
  $\BR_n^2/n$, for all $d \geq 2$.  Thus, even the best linear
  smoothers have the same poor performance as the mean over
  \smash{$\cT_d(\BR_n)$}.

\item We show in Section \ref{sec:tv_minimax} that the TV estimator is
  (essentially) minimax optimal over \smash{$\cT_d(\BR_n)$}, as the
  minimax risk over this class scales as $\BR_n/n$ (ignoring log terms).
\end{itemize}

To summarize, these results reveal a significant gap
between linear smoothers and optimal estimators like TV denoising, for
estimation over \smash{$\cT_d(\BR_n)$} in $d$
dimensions, with $d \geq 2$, as long as $\BR_n$ scales appropriately.
Roughly speaking, the TV classes encompass a
challenging setting for estimation because they are very broad,
containing a wide array of
functions---both globally smooth functions, said to have
homogeneous smoothness, and functions with vastly different
levels of smoothness at different grid locations, said to
have heterogeneous smoothness. Linear smoothers cannot handle
heterogeneous smoothness, and only nonlinear methods can enjoy good
estimation properties over the entirety of \smash{$\cT_d(\BR_n)$}.
To reiterate, a telling example is $d=2$ with the canonical scaling
\smash{$\BR_n \asymp \sqrt{n}$}, where we see that TV denoising
achieves the optimal \smash{$1/\sqrt{n}$} rate (up to log factors),
meanwhile, the best linear smoothers have max risk that is constant
over \smash{$\cT_2(\sqrt{n})$}.  See Figure \ref{fig:onehot} for an
illustration.

\begin{figure}[htb]
  \centering
  \begin{tabular}{cc}
    Trivial scaling, $\BR_n \asymp 1$ & Canonical scaling,
    $\BR_n \asymp \sqrt{n}$ \smallskip \\
    \includegraphics[width=0.45\textwidth]{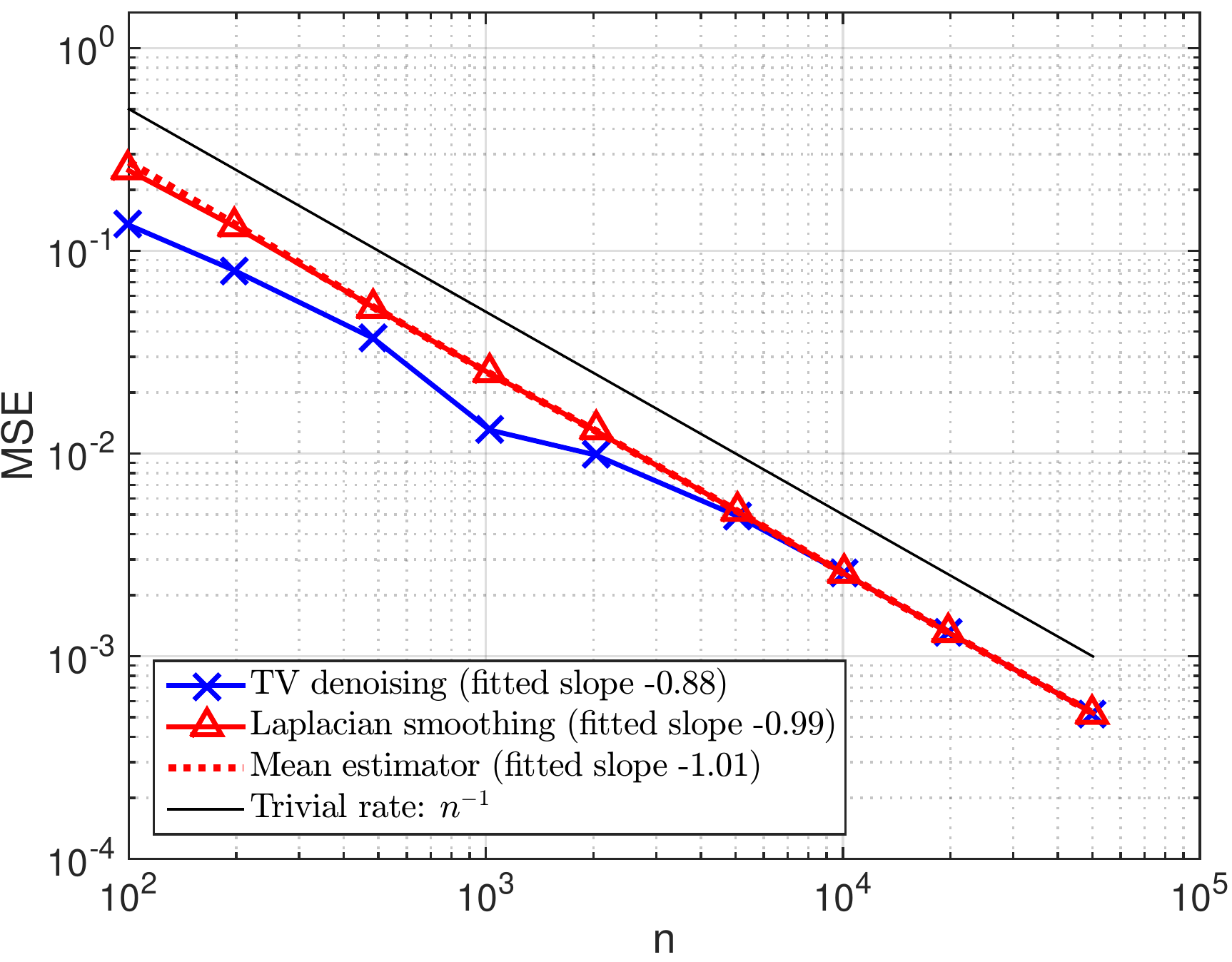} &
    \includegraphics[width=0.45\textwidth]{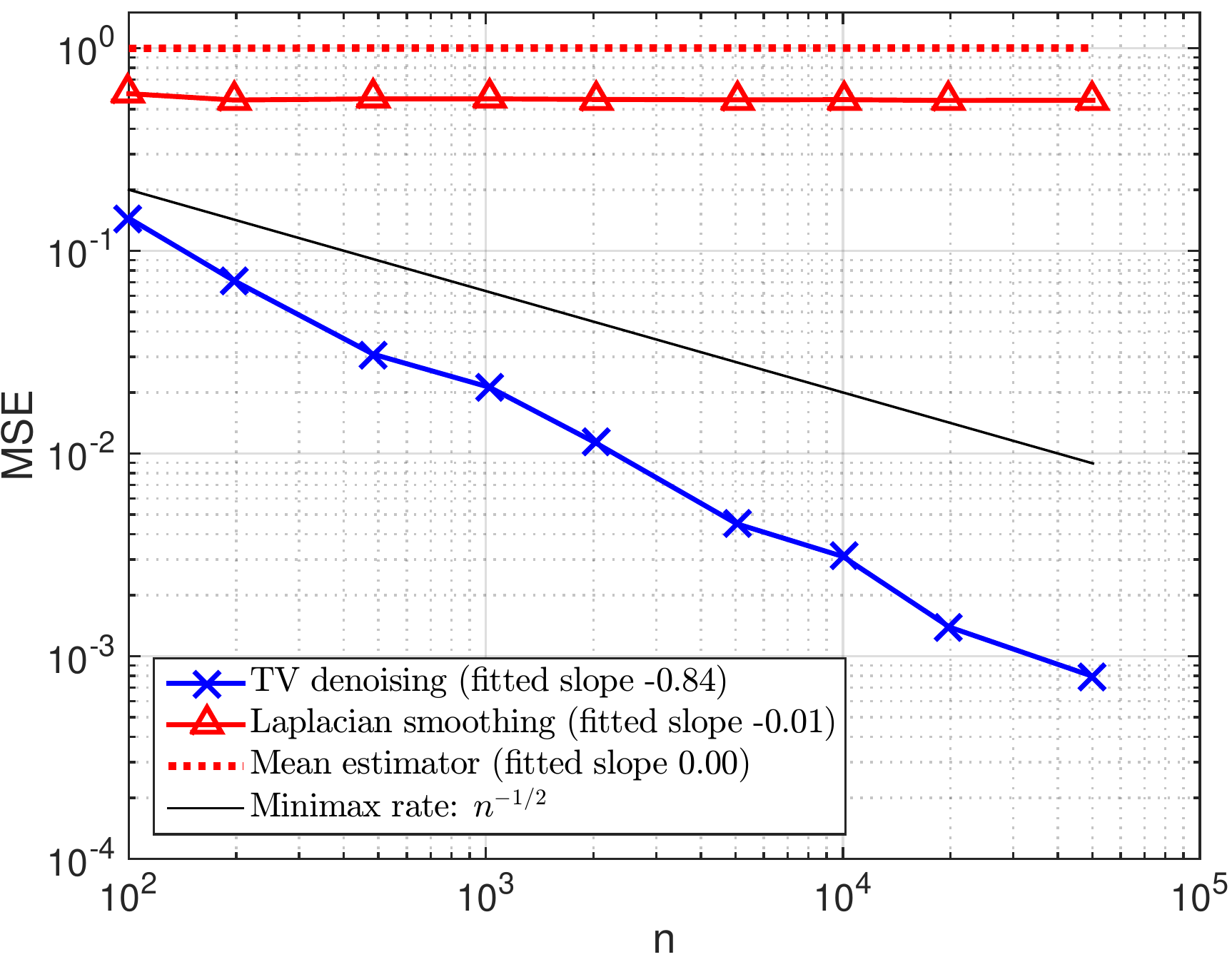}
  \end{tabular}
  \vspace{-5pt}
  \caption{\small\it MSE curves for estimation over a
    2d grid, under two very different scalings
    of $\BR_n$: constant and \smash{$\sqrt{n}$}.
    The parameter $\theta_0$ was a ``one-hot'' signal, with
    all but one component equal to 0. For each $n$, the results were
    averaged over 5 repetitions, and Laplacian smoothing and TV
    denoising were tuned for optimal average MSE.}
  \label{fig:onehot}
  \vspace{-8pt}
\end{figure}

\paragraph{Minimax rates over smaller function spaces, and
	adaptivity.}

Sections \ref{sec:sobolev_class} and \ref{sec:adaptivity} are focused
on different function spaces, discrete Sobolev spaces, which are
$\ell_2$ analogs of discrete TV spaces as we have defined them in
\eqref{eq:tv_class}.  Under the canonical scaling
of Section \ref{sec:scaling}, Sobolev spaces are
contained in TV spaces, and the former can be roughly thought of as
containing functions of more homogeneous smoothness.
The story now is more optimistic for linear smoothers, and the
following is a summary.

\begin{itemize}
\item In Section \ref{sec:sobolev_class}, we derive minimax rates
  for Sobolev spaces, and prove that
  linear smoothers---in particular, Laplacian smoothing and Laplacian
  eigenmaps---are optimal over these spaces.

\item In Section \ref{sec:adaptivity}, we discuss an interesting
  phenomenon, a phase transition of sorts, at $d=3$ dimensions.  When
  $d=1$ or $2$, the minimax rates for a TV space and its inscribed
  Sobolev space match; when $d \geq 3$, they do not, and the
  inscribed Sobolev space has a faster minimax rate.
  Aside from being an interesting statement about the TV and Sobolev
  function spaces in high dimensions, this raises an important
  question of adaptivity over the smaller Sobolev function spaces.
  As the minimax rates match for $d=1$ and $2$, any method
  optimal over TV spaces in these dimensions, such as TV denoising, is
  automatically optimal over the inscribed Sobolev spaces. But the
  question remains open for $d \geq 3$---does, e.g., TV denoising
  adapt to the faster minimax rate over Sobolev spaces?  We present
  empirical evidence to suggest that this may be true, and leave a
  formal study to future work.
\end{itemize}

\paragraph{Other considerations and extensions.}  There are many
problems related to the one that we study in this paper.  Clearly,
minimax rates for the TV and Sobolev classes over general graphs, not
just $d$-dimensional grids, are of interest. Our minimax lower bounds
for TV classes actually apply to generic graphs with bounded max
degree, though it is unclear whether to what extent they are sharp
beyond grids; a detailed study will be left to future work.
Another related topic is that of higher-order
smoothness classes, i.e., classes
contaning functions whose {\it derivatives} are (say) of bounded
variation.  The natural extension of TV denoising here is
called {\it trend filtering}, defined via the regularization of
discrete higher-order derivatives.  In the 1d setting, minimax rates,
the optimality of trend filtering, and the
suboptimality of linear smoothers is already
well-understood \citep{trendfilter}.  Trend filtering has been
defined and studied to some extent on general graphs \citep{graphtf},
but no notions of optimality have been investigated beyond 1d.  This
will also be left to future work.  Lastly, it is worth mentioning that
there are other estimators (i.e., other than the ones we study in
detail) that attain or nearly attain minimax rates over various
classes we consider in this paper. E.g., 
wavelet denoising is known to be optimal over TV classes in 1d
\citep{minimaxwave}; and comparing recent upper bounds
from \cite{needell2013stable,hutter2016optimal} with the lower bounds
in this work, we see that wavelet denoising is also nearly
minimax in 2d (ignoring log terms).

\section{Analysis over TV classes}
\label{sec:tv_class}


\subsection{Canonical scaling for TV and Sobolev classes}
\label{sec:scaling}

We start by establishing what we call a ``canonical'' scaling for the
radius $\BR_n$ of the TV ball \smash{$\cT_d(\BR_n)$} in
\eqref{eq:tv_class}, as well as the radius $\BR'_n$ of the Sobolev
ball \smash{$\cS_d(\BR'_n)$}, defined as
\begin{equation}
\label{eq:sobolev_class}
\cS_d(\BR_n') = \big\{ \theta : \|D\theta\|_2 \leq \BR_n' \big\}.
\end{equation}
Proper scalings for \smash{$\BR_n, \BR_n'$} will be critical for
properly interpreting our new results in $d$ dimensions, in a way that
is comparable to known results for $d=1$ (which are usually stated in
terms of the 1d scalings
\smash{$\BR_n \asymp 1$}, \smash{$\BR_n' \asymp 1/\sqrt{n}$}). To
study \eqref{eq:tv_class}, \eqref{eq:sobolev_class}, it helps to
introduce a third function space,
\begin{equation}
\label{eq:holder_class}
\cH_d(1) = \Big\{ \theta : \theta_i = f(i_1/\ell\ldots,i_d/\ell), \;
i=1,\ldots,n, \;
\text{for some $f \in \cH^{\mathrm{cont}}_d(1)$} \Big\}.
\end{equation}
Above, we have mapped each location $i$ on the grid to
a multi-index \smash{$(i_1,\ldots,i_d) \in \{1,\ldots,\ell\}^d$}, where
\smash{$\ell=n^{1/d}$}, and \smash{$\cH_d^{\mathrm{cont}}(1)$} denotes
the (usual) continuous Holder space on \smash{$[0,1]^d$}, i.e.,
functions that are 1-Lipschitz with respect to the $\ell_\infty$ norm.
We seek an embedding that is analogous
to the embedding of continuous Holder, Sobolev, and total variation
spaces in 1d functional analysis, namely,
\begin{equation}
\label{eq:embedding}
\cH_d(1) \subseteq \cS_d(\BR_n') \subseteq \cT_d(\BR_n).
\end{equation}
Our first lemma provides a choice of $\BR_n,\BR_n'$ that makes the
above true. Its proof, as with all proofs in this paper, can be found
in the appendix.

\begin{lemma}
\label{lem:scaling}
For $d \geq 1$, the embedding in \eqref{eq:embedding} holds with
choices \smash{$\BR_n \asymp n^{1-1/d}$} and
\smash{$\BR_n' \asymp n^{1/2-1/d}$}. Such choices are called the
{\em canonical scaling} for the function classes in
\eqref{eq:tv_class}, \eqref{eq:sobolev_class}.
\end{lemma}

As a sanity check, both the (usual) continuous Holder and Sobolev
function spaces in $d$ dimensions are known to have minimax risks that
scale as \smash{$n^{-2/(2+d)}$}, in a standard nonparametric
regression setup (e.g., \cite{gyorfi2002book}).
Under the canonical scaling \smash{$\BR_n' \asymp n^{1/2-1/d}$}, our
results in Section \ref{sec:sobolev_class} show that the discrete
Sobolev class \smash{$\cS_d(n^{1/2-1/d})$} also admits a minimax rate
of \smash{$n^{-2/(2+d)}$}.


\subsection{Minimax rates over TV classes}
\label{sec:tv_minimax}

The following is a lower bound for the minimax risk of the TV class
\smash{$\cT_d(\BR_n)$} in \eqref{eq:tv_class}.


\begin{theorem}
\label{thm:tv_minimax}
Assume $n \geq 2$, and denote \smash{$d_{\max}=2d$}. Then, for
constants $c>0$, $\rho_1 \in (2.34, 2.35)$,
\begin{equation}
\label{eq:minimax_dd}
R(\cT_d(\BR_n)) \geq
c \cdot
\begin{cases}
  \displaystyle
  \frac{\sigma \BR_n \sqrt{1+\log(\sigma d_{\max} n/\BR_n)}}
  {d_{\max} n} &
  \text{if $\BR_n \in [\sigma d_{\max} \sqrt{\log n},
\sigma d_{\max} n / \sqrt{\rho_1}]$} \\
  \BR_n^2 / (d_{\max}^2 n) \vee \sigma^2/n &
  \text{if $\BR_n < \sigma d_{\max} \sqrt{\log n}$} \\
  \sigma^2 / \rho_1 &
  \text{if $\BR_n >  \sigma d_{\max} n/\sqrt{\rho_1}$}
\end{cases}.
\end{equation}
\end{theorem}

The proof uses a simplifying reduction of the TV class, via
\smash{$\cT_d(\BR_n) \supseteq
B_1(\BR_n/d_{\max})$}, the latter set denoting the $\ell_1$ ball
of radius \smash{$\BR_n/d_{\max}$} in $\R^n$. It then invokes a sharp
characterization of the minimax risk in normal means problems over
$\ell_p$ balls due to \cite{birge2001gaussian}.  Several remarks are
in order.

\begin{remark}
The first line on the right-hand side in \eqref{eq:minimax_dd} often
provides the most useful lower bound.  To see this, recall that under
the canonical scaling for TV classes, we have
\smash{$\BR_n=n^{1-1/d}$}.  For all $d \geq 2$, this certainly implies
\smash{$\BR_n \in [\sigma d_{\max} \sqrt{\log n},
\sigma d_{\max} n / \sqrt{\rho_1}]$}, for large $n$.
\end{remark}

\begin{remark}
Even though its construction is very simple, the lower bound
on the minimax risk in \eqref{eq:minimax_dd} is sharp or nearly sharp
in many interesting cases.  Assume that
\smash{$\BR_n \in [\sigma d_{\max} \sqrt{\log n},
\sigma d_{\max} n / \sqrt{\rho_1}]$}.  The lower bound rate is
\smash{$\BR_n \sqrt{\log(n/\BR_n)}/n$}.
When $d=2$, we see that this
is very close to the upper bound rate of \smash{$\BR_n \log{n}/n$}
achieved by the TV denoiser, as stated in \eqref{eq:tv_hd}.
These two differ by at most a \smash{$\log{n}$} factor (achieved when
$\BR_n \asymp n$).  When $d \geq 3$, we see that the lower bound rate
is even closer to the upper bound rate of
\smash{$\BR_n \sqrt{\log{n}}/n$}
achieved by the TV denoiser, as in \eqref{eq:tv_hd}.  These two now
differ by at most a \smash{$\sqrt{\log{n}}$} factor (again achieved
when $\BR_n \asymp n$).  We hence conclude that the TV denoiser is
essentially minimax optimal in all dimensions $d \geq 2$.
\end{remark}

\begin{remark}
When $d=1$, and (say) \smash{$\BR_n \asymp 1$}, the lower bound rate of
\smash{$\sqrt{\log{n}}/n$} given by Theorem \ref{thm:tv_minimax} is not
sharp; we know from \cite{minimaxwave} (recall \eqref{eq:minimax_1d})
that the minimax rate over \smash{$\cT_1(1)$} is \smash{$n^{-2/3}$}.
The result in the theorem (and also Theorem
\ref{thm:tv_minimax_linear}) in fact holds more generally, beyond
grids: for an arbitrary graph $G$, its edge incidence matrix $D$,
and \smash{$\cT_d(\BR_n)$} as defined in \eqref{eq:tv_class}, the
result holds for \smash{$d_{\max}$} equal to the max degree of $G$. It
is unclear to what extent this is sharp, for different graph models.
\end{remark}

\subsection{Minimax linear rates over TV classes}
\label{sec:tv_minimax_linear}

We now turn to a lower bound on the minimax linear risk of the TV
class \smash{$\cT_d(\BR_n)$} in \eqref{eq:tv_class}.

\begin{theorem}
\label{thm:tv_minimax_linear}
Recall the notation \smash{$d_{\max}=2d$}. Then
\begin{equation}
\label{eq:minimax_dd_linear}
  R_L(\cT_d(\BR_n)) \geq
  \frac{\sigma^2 \BR_n^2}{\BR_n^2 + \sigma^2 d_{\max}^2 n}
  \vee \frac{\sigma^2}{n} \geq
\frac{1}{2} \Bigg( \frac{\BR_n^2}{d_{\max}^2n} \wedge \sigma^2 \Bigg)
\vee \frac{\sigma^2}{n}.
\end{equation}
\end{theorem}

The proof relies on an elegant meta-theorem on minimax rates
from \cite{donoho1990minimax}, which uses the concept of
a ``quadratically convex'' set, whose minimax linear risk is the same
as that of its hardest rectangular subproblem.  An alternative proof
can be given entirely from first principles.  See the appendix.

\begin{remark}
\label{rem:tv_minimax_linear}
When \smash{$\BR_n^2$} grows with $n$, but not too fast (scales as
\smash{$\sqrt{n}$}, at most), the lower bound rate in
\eqref{eq:minimax_dd_linear} will be
\smash{$\BR_n^2/n$}.  Compared to the \smash{$\BR_n/n$} minimax rate
from Theorem \ref{thm:tv_minimax} (ignoring log terms), we see a clear
gap between optimal nonlinear and linear estimators.  In fact, under
the canonical scaling \smash{$\BR_n \asymp n^{1-1/d}$}, for
any $d \geq 2$, this gap is seemingly huge: the lower bound for the
minimax linear rate will be a constant, whereas the minimax rate from
Theorem \ref{thm:tv_minimax} (ignoring log terms) will be
\smash{$n^{-1/d}$}.
\end{remark}

We now show that the lower bound in Theorem
\ref{thm:tv_minimax_linear} is essentially tight, and remarkably, it
is certified by analyzing two trivial linear estimators: the mean
estimator and the identity estimator.

\begin{lemma}
\label{lem:mean}
Let $M_n$ denote the largest column norm of \smash{$D^\dagger$}.
For the mean estimator
\smash{$\htheta^{\mathrm{mean}} = \bar{y} \1$},
\begin{equation*}
\sup_{\theta_0 \in \cT_d(\BR_n)}
\E \big[\MSE(\htheta^{\mathrm{mean}}, \theta_0) \big]
\leq \frac{\sigma^2 + \BR_n^2 M_n^2}{n},
\end{equation*}
From Proposition 4 in \cite{hutter2016optimal}, we have
\smash{$M_n=O(\sqrt{\log{n}})$}
when $d=2$ and $M_n=O(1)$ when $d \geq 3$.
\end{lemma}

The risk of the identity estimator
\smash{$\htheta^{\mathrm{id}} =  y$} is clearly $\sigma^2$.
Combining this logic
with Lemma \ref{lem:mean} gives the upper bound
\smash{$R_L(\cT_d(\BR_n)) \leq (\sigma^2+\BR_n^2M_n^2)/n
\wedge \sigma^2$}. Comparing this with the lower bound
described in Remark \ref{rem:tv_minimax_linear}, we see that the two
rates basically match, modulo the $M_n^2$ factor in the upper bound,
which only provides an extra \smash{$\log{n}$} factor when $d=2$.
The takeaway message: in the sense of max risk, the best
linear smoother does not perform much better than the trivial
estimators.

Additional empirical experiments, similar to those shown
in Figure \ref{fig:onehot}, are given in the appendix.


\section{Analysis over Sobolev classes}
\label{sec:sobolev_class}

Our first result here is a lower bound on the minimax risk
of the Sobolev class \smash{$\cS_d(\BR_n')$} in
\eqref{eq:sobolev_class}.

\begin{theorem}
\label{thm:sobolev_minimax}
For a universal constant $c>0$,
\begin{equation*}
R(\cS_d(\BR_n')) \geq \frac{c}{n} \Big(
(n\sigma^2)^\frac{2}{d+2} (\BR_n')^\frac{2d}{d+2} \wedge n
\sigma^2 \wedge n^{2/d}(\BR_n')^2 \Big) + \frac{\sigma^2}{n}.
\end{equation*}
\end{theorem}

Elegant tools for minimax analysis from \cite{donoho1990minimax},
which leverage the fact that the ellipsoid \smash{$\cS_d(\BR_n')$} is
orthosymmetric and quadratically convex (after a rotation), are used
to prove the result.

The next theorem gives upper bounds, certifying
that the above lower bound is tight, and showing that Laplacian
eigenmaps and Laplacian smoothing, both linear smoothers, 
are optimal over \smash{$\cS_d(\BR_n')$} for all $d$ and for $d=1,2,$ 
or $3$ respectively. 

\begin{theorem}
\label{thm:le_ls}
For Laplacian eigenmaps, \smash{$\htheta^{\mathrm{LE}}$} in
\eqref{eq:le}, with \smash{$k \asymp ((n (\BR'_n)^d)^{2/(d+2)} \vee 1) 
  \wedge n$}, we have
\begin{equation*}
\sup_{\theta_0 \in \cS_d(\BR_n')}
\E \big[\MSE(\htheta^{\mathrm{LE}}, \theta_0)\big] \leq
\frac{c}{n} \Big( (n\sigma^2)^{\frac{2}{d+2}}
(\BR_n')^{\frac{2d}{d+2}} \wedge n \sigma^2
\wedge n^{2/d}(\BR_n')^2 \Big) + \frac{c\sigma^2}{n},
\end{equation*}
for a universal constant $c>0$, and $n$ large enough.
When $d=1,2,$ or $3$, the same bound holds for Laplacian smoothing 
\smash{$\htheta^{\mathrm{LS}}$} in \eqref{eq:le}, with
\smash{$\lambda \asymp (n/(\BR'_n)^2)^{2/(d+2)}$} (and a
possibly different constant $c$).
\end{theorem}

\begin{remark}
As shown in the proof, Laplacian smoothing is nearly minimax
rate optimal over \smash{$\cS_d(\BR_n')$} when $d=4$, just incurring
an extra log factor. It is unclear to us whether this method is
still (nearly) optimal when $d \geq 5$; based on insights from
our proof technique, we conjecture that it is not.
\end{remark}



\section{A phase transition, and adaptivity}
\label{sec:adaptivity}

The TV and Sobolev classes in \eqref{eq:tv_class} and
\eqref{eq:sobolev_class}, respectively, display a curious
relationship.  We reflect on Theorems \ref{thm:tv_minimax} and
\ref{thm:sobolev_minimax}, under the canonical scaling with
\smash{$\BR_n \asymp n^{1-1/d}$} and
\smash{$\BR_n' \asymp n^{1/2-1/d}$}, that, recall, guarantees
\smash{$\cS_d(\BR_n') \subseteq \cT_d(\BR_n)$}.  (This is done for
concreteness; statements could also be made outside of this case,
subject to an appropriate relationship with
\smash{$\BR_n/\BR_n' \asymp \sqrt{n}$}.)
When $d=1$, both the TV and Sobolev classes have a
minimax rate of \smash{$n^{-2/3}$} (this TV result is
actually due to \cite{minimaxwave}, as stated in
\eqref{eq:minimax_1d}, not Theorem \ref{thm:tv_minimax}).  When $d=2$,
both the TV and Sobolev classes again have the same minimax rate of
\smash{$n^{-1/2}$}, the caveat being that the rate for TV class has an
extra \smash{$\sqrt{\log{n}}$} factor.  But for all $d \geq 3$, the
rates for the canonical TV and Sobolev classes differ, and
the smaller Sobolev spaces have faster rates than their inscribing TV
spaces.
This may be viewed as a phase transition at
$d=3$; see Table \ref{tab:phase_transition}.

\renewcommand\arraystretch{1.25}
\begin{table}[t]
\centering
\begin{tabular}{c|ccc}
  Function class & Dimension 1 & Dimension 2 & Dimension $d \geq 3$
  \\ \hline
  TV ball \smash{$\cT_d(n^{1-1/d})$} & $n^{-2/3}$ &
  $n^{-1/2}\sqrt{\log{n}}$ & $n^{-1/d}\sqrt{\log{n}}$ \\
  Sobolev ball \smash{$\cS_d(n^{1/2-1/d})$} & $n^{-2/3}$
  & $n^{-1/2}$ & $n^{-\frac{2}{2+d}}$
  \end{tabular}
  \smallskip
  \caption{\small\it Summary of rates for canonically-scaled TV
    and Sobolev spaces.}
  \label{tab:phase_transition}
\end{table}
\renewcommand\arraystretch{1}

\begin{figure}[t]
\centering
\begin{tabular}{cc}
Linear signal in 2d & Linear signal in 3d \\
{\includegraphics[width=0.45\textwidth]{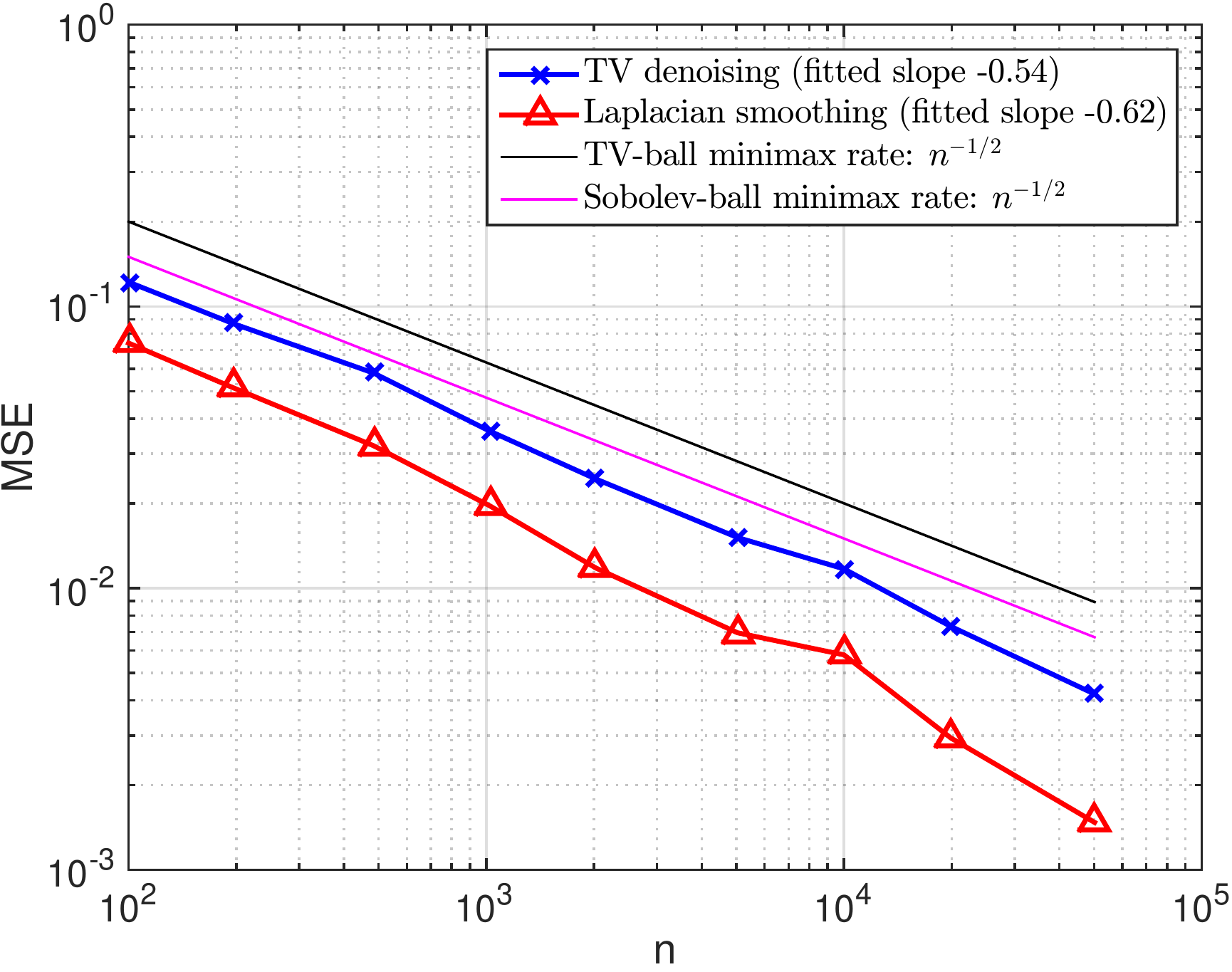}} &
{\includegraphics[width=0.45\textwidth]{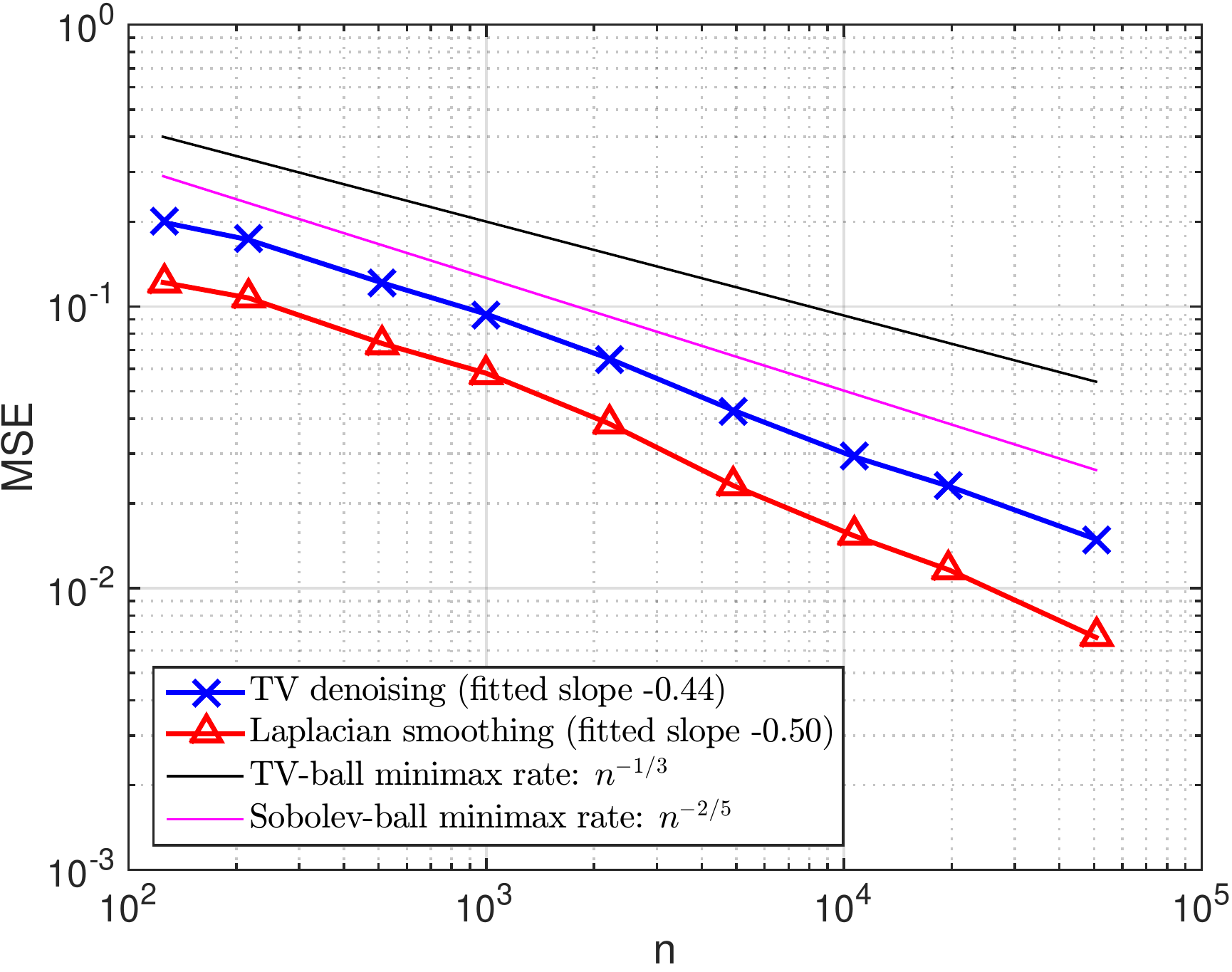}}
\end{tabular}
\vspace{-5pt}
\caption{\small\it MSE curves for estimating a ``linear'' signal, a
  very smooth signal, over 2d and 3d grids. For each $n$, the
  results were averaged over 5 repetitions, and  Laplacian smoothing
  and TV denoising were tuned for best average MSE performance.  The
  signal was set to satisfy
  \smash{$\|D\theta_0\|_2 \asymp n^{1/2-1/d}$},
    matching the canonical scaling.}
\label{fig:linear_signal}
\vspace{-8pt}
\end{figure}

We may paraphrase to say that 2d is just like 1d, in that expanding
the Sobolev ball into a larger TV ball does not hurt the minimax rate,
and methods like TV denoising are automatically {\it adaptive}, i.e.,
optimal over both the bigger and smaller classes.  However,
as soon as we enter the 3d world, it is no longer clear whether TV
denoising can adapt to the smaller, inscribed Sobolev ball, whose
minimax rate is faster, \smash{$n^{-2/5}$} versus \smash{$n^{-1/3}$}
(ignoring log factors). Theoretically, this is an interesting open
problem that we do not approach in this paper and leave to future
work.

We do, however, investigate the matter empirically: see Figure
\ref{fig:linear_signal}, where we run Laplacian smoothing and TV
denoising on a highly smooth ``linear'' signal $\theta_0$.  This
is constructed so that each component $\theta_i$ is proportional to
$i_1+i_2+\ldots+i_d$ (using the multi-index notation
$(i_1,\ldots,i_d)$ of \eqref{eq:holder_class} for grid location $i$),
and the Sobolev norm is \smash{$\|D\theta_0\|_2 \asymp
  n^{1/2-1/d}$}.
Arguably, these are among the ``hardest'' types of functions for TV
denoising to handle.  The left panel, in 2d, is a case in which we
know that TV denoising attains the minimax rate; the right panel, in
3d, is a case in which we do not, though empirically, TV denoising
surely seems to be doing better than the slower minimax rate of
\smash{$n^{-1/3}$} (ignoring log terms) that is associated with the
larger TV ball.

Even if TV denoising is shown to be
minimax optimal over the inscribed Sobolev balls when $d \geq 3$, note
that this does not necessarily mean that we should scrap Laplacian
smoothing in favor of TV denoising, in all problems.  Laplacian
smoothing is the unique Bayes estimator in a normal means model under
a certain Markov random field prior (e.g., \cite{sharpnack2010});
statistical decision theory therefore tells that it is
{\it admissible}, i.e., no other estimator---TV denoising
included---can uniformly dominate it.


\section{Discussion}
\label{sec:discussion}

We conclude with a quote from Albert Einstein: ``Everything
should be made as simple as possible, but no simpler''.
In characterizing the minimax rates for TV classes,
defined over $d$-dimensional grids, we have shown that simple methods
like Laplacian smoothing and Laplacian eigenmaps---or even in fact,
all linear estimators---must be passed up in favor of
more sophisticated, nonlinear estimators, like TV denoising, if one
wants to attain the optimal max risk.
Such a result was previously known when $d=1$; our work
has extended it to all dimensions $d \geq 2$.
We also characterized
the minimax rates over discrete Sobolev classes, revealing an
interesting phase transition where the optimal rates over TV and
Sobolev spaces, suitably scaled, match when $d=1$ and $2$ but diverge
for $d \geq 3$. It is an open question as to whether an estimator like
TV denoising can be optimal over both spaces, for all $d$.


\appendix

\section{Proofs}

We present proofs of all results, according to the order in which they
appear in the paper.  

\subsection{Proof of Lemma \ref{lem:scaling}
  (canonical scaling)}  

Suppose that \smash{$\theta \in \cH_d(1)$} that is a discretization of
a 1-Lipshitz function $f$, i.e.,
\smash{$\theta_i = f(i_1/\ell\ldots,i_d/\ell)$}, $i=1,\ldots,n$.
We first we compute and bound its squared Sobolev norm
\begin{align*}
\|D \theta\|_2^2 = \sum_{(i,j) \in E} (\theta_i - \theta_j)^2 &=
\sum_{(i,j) \in E} \big( f(i_1/\ell,\ldots,i_d/\ell) -
f(j_1/\ell,\ldots,j_d/\ell)\big)^2 \\
&\leq \sum_{(i,j) \in E} \big\|
(i_1/\ell,\ldots,i_d/\ell) -
(j_1/\ell,\ldots,j_d/\ell) \big\|_\infty^2 \\
&= m/\ell^2,
\end{align*}
where, recall, we denote by $m=|E|$ the number of edges in the grid.
In the second line we used the 1-Lipschitz property of
$f$, and in the third we used that multi-indices corresponding to
adjacent locations on the grid are exactly 1 apart, in $\ell_\infty$
distance. Thus we see that setting
\smash{$\BR_n'=\sqrt{m}/\ell$} gives the desired containment
\smash{$\cS_d(\BR_n') \supseteq \cH_d(1)$}.  It is always true
that $m \asymp n$ for a $d$-dimensional grid (though the constant may
depend on $d$), so that \smash{$\sqrt{m}/\ell \asymp n^{1/2-1/d}$}.
This completes the proof for the Sobolev class scaling.

As for TV class scaling, the result follows from the simple
fact that \smash{$\|x\|_1 \leq \sqrt{m} \|x\|_2$} for any $x \in
\R^m$, so that we may take \smash{$\BR_n = \sqrt{m} \BR_n' =
  n^{1-1/d}$}.
\hfill\qedsymbol

\subsection{Proof of Theorem \ref{thm:tv_minimax} 
(minimax rates over TV classes)}

Here and henceforth,
we use the notation \smash{$B_p(r) = \{ x : \|x\|_p \leq r \}$} for
the $\ell_p$ ball of radius $r$, where $p,r > 0$ (and the ambient
dimension will be determined based on the context).

We begin with a very simple lemma, that embeds an $\ell_1$ ball inside
the TV ball \smash{$\cT_d(\BR_n)$}.

\begin{lemma}
\label{lem:subset}
Let $G$ be a graph with maximum degree \smash{$d_{\max}$}, and
let $D\in\R^{m\times n}$ be its incidence matrix. Then for any
$r>0$, it holds that \smash{$B_1(r/d_{\max}) \subseteq T_d(r)$}.
\end{lemma}

\begin{proof}
Write $D_i$ for the $i$th column of $D$.
The proof follows from the observation that, for any $\theta$,
\begin{equation*}
\|D\theta\|_1
= \bigg\| \sum_{i=1}^n D_i \theta_i \bigg\|_1
\leq \sum_{i=1}^n \|D_i\|_1 |\theta_i|
\leq \Big( \max_{i=1,\ldots,n} \|D_i\|_1\Big)
\|\theta\|_1 = d_{\max} \|\theta\|_1.
\end{equation*}
\end{proof}

To prove Theorem \ref{thm:tv_minimax}, we will rely on a result from
\citet{birge2001gaussian}, which gives a lower bound for the risk in a
normal means problem, over $\ell_p$ balls. Another related, earlier
result is that of \citet{donoho1994minimax}; however, the Birge and
Massart result places no restrictions on the radius of the ball in
question, whereas the Donoho and Johnstone result does. Translated
into our notation, the Birge and Massart result is as follows.

\begin{lemma}[Proposition 5 of \citet{birge2001gaussian}]
\label{lem:birge}
Assume i.i.d.\ observations \smash{$y_i \sim
  N(\theta_{0,i},\sigma^2)$}, $i=1,\ldots,n$, and $n \geq 2$.
Then the minimax risk over the $\ell_p$ ball $B_p(r_n)$, where
$0 < p < 2$, satisfies
\begin{equation*}
n \cdot R(B_p(r_n)) \geq
c \cdot
\begin{cases}
\displaystyle
\sigma^{2-p} r_n^p
\bigg[1+\log\bigg(
\frac{\sigma^p n}{r_n^p}\bigg)
\bigg]^{1-p/2} &
\text{if $\sigma\sqrt{\log n}\leq r_n \leq
  \sigma n^{1/p}/\sqrt{\rho_p}$} \\
r_n^2 &\text{if $ r_n< \sigma \sqrt{\log n}$} \\
\sigma^2 n / \rho_p &\text{if
$r_n > \sigma n^{1/p} / \sqrt{\rho}$}
\end{cases}.
\end{equation*}
Here $c>0$ is a universal constant, and $\rho_p > 1.76$ is the unique
solution of $\rho_p\log\rho_p = 2/p$.
\end{lemma}

Finally, applying Lemma \ref{lem:birge} to
\smash{$B_1(\BR_n/d_{\max})$} almost gives the lower bound as stated
in Theorem \ref{thm:tv_minimax}. However, note that the
minimax risk in question is trivially lower bounded by $\sigma^2/n$,
because
\begin{align*}
\inf_{\htheta} \; \sup_{\theta_0 \in \cT_d(\BR_n)} \frac{1}{n} \;
\E\|\htheta - \theta_0\|_2^2 &\geq
\inf_{\htheta} \; \sup_{\theta_0 : \theta_{0,1} = \ldots =
  \theta_{0,n}} \;
\frac{1}{n} \sum_{i=1}^n \E (\htheta_i -\theta_{0,1})^2 \\
&= \inf_{\htheta_1} \;
\sup_{\theta_{0,1}} \; \E(\htheta_1-\theta_{0,1})^2 \\
&= \frac{\sigma^2}{n}.
\end{align*}
In the second to last line, the problem is to estimate a 1-dimensional
mean parameter $\theta_{0,1}$, given the observations $y_i \sim
N(\theta_{0,1}, \sigma^2)$, i.i.d., for $i=1,\ldots,n$; this has a
well-known minimax risk of $\sigma^2/n$. What this means for our TV
problem: to derive a
lower bound for the minimax rate over \smash{$\cT_d(\BR_n)$}, we may
take the maximum of the result of applying Lemma \ref{lem:birge} to
\smash{$B_1(\BR_n/d_{\max})$} and
$\sigma^2/n$. One can see that
the term \smash{$\sigma^2/n$} only plays a role for small
\smash{$\BR_n$}, i.e., it effects the case when
\smash{$\BR_n < \sigma d_{\max} \sqrt{\log n}$}, where the lower bound
becomes \smash{$\BR_n^2 /(d_{\max}^2 n) \vee \sigma^2/n$}.
\hfill\qedsymbol

\subsection{Proof of Theorem \ref{thm:tv_minimax_linear} 
  (minimax linear rates over TV classes)}  

First we recall a few definitions, from \citet{donoho1990minimax}.
Given a set $A \subseteq \R^k$, its {\it quadratically convex hull}
\smash{$\mathrm{qconv}(A)$} is defined as
\begin{gather*}
\mathrm{qconv}(A) = \big\{ (x_1,\ldots,x_k) : (x_1^2,\ldots,x_k^2) \in
\mathrm{conv}(A_+^2) \big\}, \quad \text{where} \\
A_+^2 = \big\{ (a_1^2,\ldots,a_k^2) : a \in A, \; a_i \geq 0, \;
i=1,\ldots,k \big\}.
\end{gather*}
(Here \smash{$\mathrm{conv}(B)$} denotes the convex hull of a set
$B$.) Furthermore, the set $A$ is called {\it quadratically convex}
provided that \smash{$\mathrm{qconv}(A) = A$}.  Also, $A$ is called
{\it orthosymmetric} provided that $(a_1,\ldots,a_k) \in A$ implies
\smash{$(\sigma_1 a_1,\ldots, \sigma_k a_k) \in A$}, for any choice of
signs \smash{$\sigma_1,\ldots,\sigma_k \in \{-1,1\}$}.

Now we proceed with the proof. Following from equation (7.2) of
\citet{donoho1990minimax},
\begin{equation*}
\mathrm{qconv}\big(B_1(\BR_n/d_{\max})\big) =
B_2(\BR_n/d_{\max}).
\end{equation*}
Theorem 11 of \citet{donoho1990minimax} states that, for
orthosymmetric, compact sets, such as
\smash{$B_1(\BR_n/d_{\max})$}, the minimax linear risk equals that
of its quadratically convex hull.
Moreover, Theorem 7 of \citet{donoho1990minimax}
tells us that for sets that are orthosymmetric, compact, convex, and
quadratically convex, such as \smash{$B_2(\BR_n/d_{\max})$}, the
minimax linear
risk is the same as the minimax linear risk over the worst rectangular
subproblem.  We consider \smash{$B_\infty(\BR_n/(d_{\max}\sqrt{n}))$},
and abbreviate \smash{$r_n=\BR_n/(d_{\max}\sqrt{n})$}.  It is fruitful
to study rectangles because the problem separates across dimensions,
as in
\begin{align*}
\inf_{\htheta \, \text{linear}} \; \sup_{\theta_0 \in B_\infty(r_n)}
\; \E \bigg[ \frac{1}{n} \sum_{i=1}^n (\htheta_i - \theta_{0,i})^2
\bigg] &= \frac{1}{n} \sum_{i=1}^n \bigg[
\inf_{\htheta_i \, \text{linear}} \; \sup_{|\theta_{0,i}| \leq r_n} \;
\E(\htheta_i - \theta_{0,i})^2 \bigg] \\
& = \inf_{\htheta_1 \, \text{linear}} \;
\sup_{|\theta_{0,1}| \leq r_n} \; \E(\htheta_1 - \theta_{0,1})^2.
\end{align*}
Thus it suffices to compute the minimax linear risk over the 1d class
\smash{$\{ \theta_{0,1} : |\theta_{0,1}| \leq r_n \}$}.  It is easily
shown (e.g., see Section 2 of \citet{donoho1990minimax}) that this is
\smash{$r_n^2\sigma^2/(r_n^2+\sigma_2^2)$}, and so this is precisely
the minimax linear risk for \smash{$B_2(\BR_n/d_{\max})$}, and for
\smash{$B_1(\BR_n/d_{\max})$}.

To get the first lower bound as stated in the theorem, we simply take
a maximum of \smash{$r_n^2\sigma^2/(r_n^2+\sigma_2^2)$}
and $\sigma^2/n$, as the latter is the minimax risk for
estimating a 1-dimensional mean parameter given $n$ observations in
a normal model with variance $\sigma^2$, recall the end of
the proof of Theorem \ref{thm:tv_minimax}. To get the second, we use
the fact that \smash{$2ab/(a+b) \geq \min\{a,b\}$}.  This completes
the proof.
\hfill\qedsymbol

\subsection{Alternative proof of Theorem \ref{thm:tv_minimax_linear}} 

Here, we reprove Theorem \ref{thm:tv_minimax_linear} using elementary
arguments. We write $y=\theta_0+\epsilon$, for \smash{$\epsilon
  \sim N(0,\sigma^2 I)$}. Given an arbitary linear estimator,
\smash{$\htheta=Sy$} for a matrix $S \in \R^{n\times n}$, observe that
\begin{align}
\nonumber
  \E \big[\MSE(\htheta,\theta_0) \big]
  = \frac{1}{n} \E \| \htheta - \theta_0 \|_2^2
  &= \frac{1}{n} \E \| S(\theta_0 + \epsilon) - \theta_0 \|_2^2 \\
\nonumber
  &= \frac{1}{n} \E \| S\epsilon \|_2^2 +
  \frac{1}{n} \|(S-I) \theta_0 \|_2^2 \\
  \label{eq:bias_var}
  &=  \frac{\sigma^2}{n} \| S \|_F^2 +
\frac{1}{n} \|(S-I) \theta_0 \|_2^2,
\end{align}
which we may view as the variance and (squared) bias terms,
respectively.  Now denote by $e_i$ the $i$th standard basis vector,
and consider
\begin{align*}
 \frac{\sigma^2}{n} \| S \|_F^2 +
\bigg(\sup_{\theta_0 : \|D\theta_0\|_1 \leq \BR_n} \;
\frac{1}{n} \|(S-I) \theta_0 \|_2^2 \bigg)
  &\geq \frac{\sigma^2}{n} \|S\|_F^2 +
  \frac{\BR_n^2}{d_{\max}^2 n} \Bigg(
  \max_{i=1,\ldots,n} \; \| (I-S)e_i \|_2^2 \Bigg) \\
  &\geq \frac{\sigma^2}{n} \|S\|_F^2 +
  \frac{\BR_n^2}{d_{\max}^2 n^2}
  \sum_{i=1}^n \| (I-S)e_i \|_2^2 \\
  &= \frac{\sigma^2}{n} \|S\|_F^2 +
  \frac{\BR_n^2}{d_{\max}^2 n^2} \| (I-S) \|_F^2 \\
  &\geq \frac{\sigma^2}{n} \sum_{i=1}^n S_{ii}^2 +
  \frac{\BR_n^2}{d_{\max}^2 n^2} \sum_{i=1}^n (1-S_{ii})^2 \\
  &= \frac{1}{n} \sum_{i=1}^n \bigg(\sigma^2 S_{ii}^2 +
  \frac{\BR_n^2}{d_{\max}^2 n} (1-S_{ii})^2\bigg).
\end{align*}
Here $S_{ii}$, $i=1,\ldots,n$ denote the diagonal entries of $S$.
To bound each term in the sum, we apply the simple inequality
\smash{$ax^2 + b(1-x)^2 \geq ab/(a+b)$} for all $x$ (since a short
calculation shows that the quadratic in $x$ here is minimized at
$x=b/(a+b)$). We may continue on lower bounding the last displayed
expression, giving
\begin{equation*}
 \frac{\sigma^2}{n} \| S \|_F^2 +
\bigg(\sup_{\theta_0 : \|D\theta_0\|_1 \leq \BR_n}
\frac{1}{n} \|(S-I) \theta_0 \|_2^2 \bigg) \geq
\frac{\sigma^2 \BR_n^2}{\BR_n^2 + \sigma^2 d_{\max}^2 n}.
\end{equation*}
Lastly, we may take the maximum of this with $\sigma^2/n$ in order to
derive a final lower bound, as argued in the proof of Theorem
\ref{thm:tv_minimax_linear}.
\hfill\qedsymbol

\subsection{Proof of Lemma \ref{lem:mean} (mean
  estimator over TV classes)} 

For this estimator, the smoother matrix is $S=\1 \1^T/n$ and so
$\|S\|_F^2 = 1$. From \eqref{eq:bias_var}, we have
\begin{equation*}
  \E\big[\MSE(\htheta^{\mathrm{mean}},\theta_0)\big]
  =  \frac{\sigma^2}{n} +
  \frac{1}{n} \| \theta_0 - \bar\theta_0 \1 \|_2^2,
\end{equation*}
where \smash{$\bar\theta_0=(1/n) \sum_{i=1}^n \theta_{0,i}$}.
Now
\begin{align*}
  \sup_{\theta_0 : \|D \theta_0\|_1 \leq \BR_n}
  \frac{1}{n} \| \theta_0 - \bar\theta_0 \1 \|_2^2
  &= \sup_{x \in \row(D) : \|D x\|_1 \leq \BR_n}
   \frac{1}{n} \| x \|_2^2 \\
  &= \sup_{z \in \col(D): \|z\|_1 \leq \BR_n}  \frac{1}{n}
  \| D^\dagger z \|_2^2 \\
  &\leq \sup_{z : \|z\|_1 \leq \BR_n}
  \frac{1}{n} \| D^\dagger z \|_2^2 \\
  &= \frac{\BR_n^2}{n} \max_{i=1,\ldots,n} \;
  \| D^\dagger_i \|_2^2 \\
  &\leq \frac{\BR_n^2 M_n^2}{n},
\end{align*}
which establishes the desired bound.
\hfill\qedsymbol

\subsection{Proof of Theorem \ref{thm:sobolev_minimax}
  (minimax rates over Sobolev classes)}

Recall that we denote by $L=V \Sigma V^T$ the eigendecomposition of
the graph Laplacian $L=D^T D$, where
$\Sigma=\diag(\rho_1,\ldots,\rho_n)$ with $0=\rho_1< \rho_2 \leq
\ldots \leq \rho_n$, and where $V \in \R^{n\times n}$ has
orthonormal columns. Also denote by $D=U \Sigma^{1/2} V^T$ the
singular value decomposition of the edge incidence matrix $D$, where
$U \in \R^{m\times n}$ has orthonormal columns.\footnote{When
$d=1$, we have $m=n-1$ edges, and so it is not be possible for $U$ to
have orthonormal columns; however, we can just take its first column
to be all 0s, and take the rest as the eigenbasis for $\R^{n-1}$, and all
the arguments given here will go through.} First notice that
\begin{equation*}
\|D\theta_0\|_2 = \|U \Sigma^{1/2} V^T \theta_0\|_2 = \|\Sigma^{1/2}
V^T \theta_0\|_2.
\end{equation*}
This suggests that a rotation by $V^T$ will further simplify the
minimax risk over \smash{$\cS_d(\BR_n')$}, i.e.,
\begin{align}
\nonumber
\inf_{\htheta} \;
\sup_{\theta_0 : \|\Sigma^{1/2} V^T \theta_0\|_2 \leq \BR_n'}
\frac{1}{n} \E \|\htheta - \theta_0\|_2^2 &=
\inf_{\htheta} \;
\sup_{\theta_0 :  \|\Sigma^{1/2} V^T \theta_0\|_2 \leq \BR_n'}
\frac{1}{n} \E \|V^T \htheta - V^T \theta_0\|_2^2 \\
\label{eq:rotated}
&= \inf_{\hat\gamma} \;
\sup_{\gamma_0 : \|\Sigma^{1/2} \gamma_0\|_2 \leq \BR_n'}
\frac{1}{n} \E \|\hat\gamma - \gamma_0\|_2^2,
\end{align}
where we have rotated and now consider the new parameter
\smash{$\gamma_0 = V^T \theta_0$}, constrained to lie in
\begin{equation*}
\cE_d(\BR_n') = \bigg\{ \gamma : \sum_{i=2}^n \rho_i \gamma_i^2
\leq (\BR_n')^2 \bigg\}.
\end{equation*}
To be clear, in the rotated setting \eqref{eq:rotated} we observe
a vector
\smash{$y' = V^T y \sim N(\gamma_0, \sigma^2 I)$},
and the goal is to estimate the mean parameter $\gamma_0$.
Since there are no constraints along the first dimension, we can
separate out the MSE in \eqref{eq:rotated} into that incurred on the
first component, and all other components. Decomposing
 $\gamma_0=(\alpha_0,\beta_0) \in \R^{1 \times (n-1)}$,
with similar notation for an estimator \smash{$\hat\gamma$},
\begin{align}
\nonumber
\inf_{\hat\gamma} \;
\sup_{\gamma_0 \in \cE_d(\BR_n')} \;
\frac{1}{n} \E \|\hat\gamma - \gamma_0\|_2^2 &=
\inf_{\halpha} \;
\sup_{\alpha_0} \; \frac{1}{n} \E(\halpha - \alpha_0)^2 +
\inf_{\hbeta} \; \sup_{\beta_0 \in P_{-1} (\cE_d(\BR_n'))} \;
\frac{1}{n} \E \|\hbeta - \beta_0 \|_2^2 \\
\label{eq:decompose}
&= \frac{\sigma^2}{n} +\inf_{\hbeta} \;
\sup_{\beta_0 \in P_{-1} (\cE_d(\BR_n'))}
\frac{1}{n} \E \|\hbeta - \beta_0 \|_2^2,
\end{align}
where \smash{$P_{-1}$} projects onto all
coordinate axes but the 1st, i.e., \smash{$P_{-1}(x) = (0,x_2,\ldots,x_n)$},
and in the second line we have used the fact that the minimax risk for
estimating a 1-dimensional parameter $\alpha_0$ given an observation
$z \sim N(\alpha_0, \sigma^2)$ is simply $\sigma^2$.

Let us lower bound the
second term in \eqref{eq:decompose}, i.e.,
\smash{$R(P_{-1}(\cE_d(\BR_n')))$}.
The ellipsoid \smash{$P_{-1} ( \cE_d(\BR_n'))$} is orthosymmetric,
compact, convex, and quadratically convex, hence Theorem 7 in
\citet{donoho1990minimax} tells us that its minimax linear risk is the
minimax linear risk of its hardest rectangular subproblem.  Further,
Lemma 6 in \citet{donoho1990minimax} then tells us the minimax linear
risk of its hardest rectangular subproblem is, up to a constant
factor, the same as the minimax (nonlinear) risk of the full problem.
More precisely, Lemma 6 and Theorem 7 from \citet{donoho1990minimax}
imply
\begin{equation}
\label{eq:rect_risk}
R(P_{-1}(\cE_d(\BR_n'))) \geq \frac{4}{5} R_L(P_{-1}(\cE_d(\BR_n'))) =
\sup_{H \subseteq P_{-1}(\cE_d(\BR_n'))} R_L(H),
\end{equation}
where the supremum above is taken over all rectangular subproblems,
i.e., all rectangles $H$ contained in
\smash{$P_{-1}(\cE_d(\BR_n'))$}.

To study rectangular subproblems, it helps to reintroduce the
multi-index notation for a location $i$ on
the $d$-dimensional grid, writing this as
\smash{$(i_1,\ldots,i_d) \in  \{1,\ldots,\ell\}^d$}, where
\smash{$\ell=n^{1/d}$}.
For a parameter $2 \leq \tau \leq \ell$, we consider rectangular
subsets of the form\footnote{Here, albeit unconvential, it helps to
  index \smash{$\beta \in H(\tau) \subseteq \R^{n-1}$}
  according to components $i=2,\ldots,n$, rather than
  $i=1,\ldots,n-1$.  This is so
  that we may keep the index variable $i$ to be in correspondence with
  positions on the grid.}
\begin{gather*}
H(\tau) = \big\{ \beta \in \R^{n-1}: |\beta_i| \leq t_i(\tau), \;
i=2,\ldots,n \big\}, \quad\text{where} \\
t_i(\tau) = \begin{cases}
  \BR_n'/(\sum_{j_1,\ldots,j_d \leq \tau}
  \rho_{j_1,\ldots,j_d})^{1/2} &
\text{if $i_1,\ldots,i_d \leq \tau$} \\
0 & \text{otherwise}
\end{cases}, \quad
\text{for $i=2,\ldots,n$}.
\end{gather*}
It is not hard to check that \smash{$H(\tau) \subseteq
\{ \beta \in \R^{n-1}: \sum_{i=2}^n \rho_i \beta_i^2 \leq (\BR_n')^2
\} = P_{-1}(\cE_d(\BR_n'))$}. Then, from \eqref{eq:rect_risk},
\begin{align*}
R(P_{-1}(\cE_d(\BR_n'))) \geq \sup_\tau \, R_L(H(\tau)) &=
\sup_\tau \; \frac{1}{n} \sum_{i=1}^n
\frac{t_i(\tau)^2 \sigma^2}{t_i(\tau)^2 + \sigma^2} \\
&= \sup_\tau \; \frac{1}{n}
\frac{(\tau^d-1) \sigma^2(\BR_n')^2}{(\BR_n')^2 +
\sum_{j_1,\ldots,j_d \leq \tau} \rho_{j_1,\ldots,j_d}}.
\end{align*}
The first equality is due to the fact that the minimax risk for
rectangles decouples across dimensions, and the 1d minimax linear risk
is straightforward to compute for an interval, as argued in the proof
Theorem \ref{thm:tv_minimax_linear}; the second equality simply comes
from a short calculation following the definition of $t_i(\tau)$,
$i=2,\ldots,n$.
Applying Lemma \ref{lem:lap_eigenvalues}, on the eigenvalues of the
graph Laplacian matrix $L$ for a $d$-dimensional grid, we have that
for a constant $c>0$,
\begin{equation*}
\frac{(\tau^d-1) \sigma^2(\BR_n')^2}{(\BR_n')^2 +
\sum_{j_1,\ldots,j_d \leq \tau} \rho_{j_1,\ldots,j_d}}
 \geq
\frac{(\tau^d-1) \sigma^2 (\BR_n')^2}{(\BR_n')^2 +
c\sigma^2 \tau^{d+2}/\ell^2}
\geq \frac{1}{2}
\frac{\sigma^2 (\BR_n')^2}{(\BR_n')^2 \tau^{-d} +
c\sigma^2 \tau^2/\ell^2}.
\end{equation*}
We can choose $\tau$ to maximize the expression on the right
above, given by
\begin{equation*}
\tau^* = \bigg(\frac{\ell^2
  (\BR_n')^2}{c\sigma^2}\bigg)^{\frac{1}{d+2}}.
\end{equation*}
When \smash{$2 \leq \tau^* \leq \ell$},
this provides us with the lower bound on the minimax risk
\begin{equation}
\label{eq:case1}
R(P_{-1}(\cE_d(\BR_n'))) \geq R_L(H(\tau^*)) \geq
\frac{1}{2n} \frac{\tau^d \sigma^2 (\BR_n')^2}
{2(c\sigma^2)^{\frac{d}{d+2}} (\BR_n')^\frac{4}{d+2}
\ell^{-\frac{2d}{d+2}}} = \frac{c_1}{n}
(n\sigma^2)^{\frac{2}{d+2}} (\BR_n')^{\frac{2d}{d+2}},
\end{equation}
for a constant $c_1>0$.  When \smash{$\tau^* < 2$}, we can use
$\tau=2$ as lower bound on the minimax risk,
\begin{equation}
\label{eq:case2}
R(P_{-1}(\cE_d(\BR_n'))) \geq R_L(H(2)) \geq \frac{1}{2n}
\frac{\sigma^2 \ell^2 (\BR_n')^2}{\ell^2 (\BR_n')^2 2^{-d} +
c\sigma^2 2^2} \geq \frac{c_2}{n} \ell^2 (\BR_n')^2,
\end{equation}
for a constant $c_2>0$, where in the last inequality, we used
the fact that \smash{$\ell^2 (\BR_n')^2 \leq c\sigma^2 2^{d+2}$} (just
a constant) since we are in the case \smash{$\tau^* < 2$}.  Finally,
when \smash{$\tau^* > \ell$}, we can use $\tau=\ell$ as a lower bound
on the minimax risk,
\begin{equation}
\label{eq:case3}
R(P_{-1}(\cE_d(\BR_n'))) \geq R_L(H(\ell)) \geq \frac{1}{2n}
\frac{\sigma^2 (\BR_n')^2}{\ell^{-d} (\BR_n')^2 + c\sigma^2} \geq
c_3 \sigma^2,
\end{equation}
for a constant $c_3>0$, where in the last inequality, we
used that \smash{$c\sigma^2 \leq \ell^{-d} (\BR_n')^2$} as we are in
the case \smash{$\tau^* > \ell$}. Taking a minimum of the lower bounds
in \eqref{eq:case1}, \eqref{eq:case2}, \eqref{eq:case3}, as a way to
navigate the cases, gives us a final lower bound on
\smash{$R(P_{-1}(\cE_d(\BR_n')))$}, and completes the proof.

\subsection{Proof of Theorem \ref{thm:le_ls} (Laplacian
  eigenmaps and Laplacian smoothing over Sobolev classes)}

We will prove the results for Laplacian eigenmaps and Laplacian
separately.

\paragraph{Laplacian eigenmaps.}  The smoother matrix for this
estimator is \smash{$S_k=V_{[k]} V_{[k]}^T$}, for a tuning parameter
$k=1,\ldots,n$.  From \eqref{eq:bias_var},
\begin{equation*}
\E\big[\MSE(\htheta^{\mathrm{LE}},\theta_0)\big] =
\frac{\sigma^2}{n} k +
\frac{1}{n} \| (I - S_k) \theta_0 \|_2^2.
\end{equation*}
Now we write \smash{$k=\tau^d$}, and analyze the max risk of the
second term,
\begin{align*}
\sup_{\theta_0 : \|D\theta_0\|_2 \leq \BR_n'}
\frac{1}{n} \| (I - S_k) \theta_0 \|_2^2
&= \sup_{z : \| z \|_2 \leq \BR_n'}
\frac{1}{n} \| (I - S_k) D^\dagger z \|_2^2 \\
&= \frac{(\BR_n')^2}{n}
\sigma^2_{\max} \big((I-S_k) D^\dagger \big) \\
&\leq \frac{(\BR_n')^2}{n}
\frac{1}{4 \sin^2( \pi \tau/ (2\ell))} \\
&\leq \frac{(\BR_n')^2}{n} \frac{4 \ell^2}{\pi^2 \tau^2}.
\end{align*}
Here we denote by \smash{$\sigma_{\max}(A)$} the maximum
singular value of a matrix $A$.
The last inequality above used the simple lower bound
$\sin(x) \geq x/2$ for $x \in [0,\pi/2]$.
The earlier inequality used that
\begin{equation*}
(I - S_k) D^\dagger = (I - V_{[k]} V_{[k]}^T) V^T (\Sigma^\dag)^{1/2}
U^T  = \big[0, \ldots, 0, V_{k+1}, \ldots, V_n\big]
(\Sigma^\dag)^{1/2} U^T,
\end{equation*}
where we have kept the same notation for the singular
value decomposition of $D$ as in the proof of Theorem
\ref{thm:sobolev_minimax}.
Therefore \smash{$\sigma^2_{\max}((I - S_k) D^\dagger)$} is the
reciprocal of the $(k+1)$st smallest eigenvalue $\rho_{k+1}$ of the graph Laplacian
$L$.  For any subset $A$ of the eigenvalues $\Lambda$ of the Laplacian, with $|A|=k$, $\rho_{k+1} \geq \min \Lambda \setminus A$. 
So for a $d$-dimensional grid, 
\begin{align*}
\rho_{k+1} 
&\geq \min \; \Lambda \setminus 
\{\rho_{i_1,\cdots,i_d} : 1\leq i_j\leq \tau \text{ for all } 1\leq j\leq d\} \\
&= 4\sin^2(\pi \tau / (2\ell))
\end{align*}
where recall \smash{$\ell=n^{1/d}$}, as explained by \eqref{eq:lap_eigenvalues}, in
the proof of Lemma \ref{lem:lap_eigenvalues}.

Hence, we have established
\begin{equation*}
\sup_{\theta_0 : \|D\theta_0\|_2 \leq \BR_n'}
\E\big[\MSE(\htheta^{\mathrm{LE}},\theta_0)\big] \leq
\frac{\sigma^2}{n} + \frac{\sigma^2}{n} \tau^d +
\frac{(\BR_n')^2}{n} \frac{4 \ell^2}{\pi^2 \tau^2}.
\end{equation*}
Choosing $\tau$ to balance the two terms on the right-hand side
above results in
\smash{$\tau^* = (2 \ell \BR_n' / (\pi \sigma))^{\frac{2}{d+2}}$}.
Plugging in this choice of $\tau$, while utilizing the bounds $1 \leq
\tau \leq \ell$, very similar to the arguments given at the end of the
proof of Theorem \ref{thm:sobolev_minimax}, gives the result for
Laplacian eigenmaps.

\paragraph{Laplacian smoothing.}  The smoother matrix for this
estimator is $S_\lambda = (I+\lambda L)^{-1}$, for a tuning parameter
$\lambda \geq 0$. From \eqref{eq:bias_var},
\begin{equation*}
\E\big[\MSE(\htheta^{\mathrm{LS}},\theta_0)\big] =
\frac{\sigma^2}{n}
\sum_{i=1}^n \frac{1}{(1+\lambda \rho_i)^2} +
\frac{1}{n} \| (I - S_\lambda) \theta_0 \|_2^2.
\end{equation*}
When $d=1,2,$ or $3$, the first term upper is bounded by 
\smash{$c_1\sigma^2/n+ c_2 \sigma^2/\lambda^{d/2}$}, 
for some constants $c_1,c_2>0$, by Lemma \ref{lem:ls_variance}. 
As for the second term,
\begin{align*}
\sup_{\theta_0 : \| D \theta_0 \|_2 \leq \BR_n'}
\frac{1}{n} \| (I - S_\lambda) \theta_0 \|_2^2
&= \sup_{z : \| z \|_2 \leq \BR_n'} \| (I - S_\lambda)
D^\dagger z \|_2^2 \\
&= \frac{(\BR_n')^2}{n} \sigma^2_{\max} \big(
  (I-S_\lambda) D^\dagger \big) \\
&= \frac{(\BR_n')^2}{n} \max_{i=2,\ldots,n} \;
\bigg( 1 - \frac{1}{1+\lambda \rho_i} \bigg)^2 \frac{1}{\rho_i} \\
&= \frac{(\BR_n')^2}{n} \lambda \max_{i=2,\ldots,n} \;
\frac{\lambda \rho_i}{(1+\lambda \rho_i)^2} \\
&\leq \frac{(\BR_n')^2 \lambda}{4n}.
\end{align*}
In the third equality we have used the fact the eigenvectors of
\smash{$I-S_\lambda$} are the left singular vectors of
\smash{$D^\dagger$}, and in the last inequailty we have used the
simple upper bound \smash{$f(x) = x/(1+x)^2 \leq 1/4$} for $x \geq 0$
(this function being maximized at $x=1$).

Therefore, from what we have shown,
\begin{equation*}
\sup_{\theta_0 : \|D\theta_0\|_2 \leq \BR_n'}
\E\big[\MSE(\htheta^{\mathrm{LS}},\theta_0)\big] \leq
\frac{c_1\sigma^2}{n}+ \frac{c_2 \sigma^2}{\lambda^{d/2}} +
\frac{(\BR_n')^2 \lambda}{4n}.
\end{equation*}
Choosing $\lambda$ to balance the two terms on the right-hand
side above gives
\smash{$\lambda^* = c (n/(\BR_n')^2)^{2/(d+2)}$}, for a constant
$c>0$. Plugging in this choice, and using upper bounds
from the trivial cases $\lambda=0$ and $\lambda=\infty$ when
\smash{$\BR_n'$} is very small or very large, respectively, gives the
result for Laplacian smoothing.
\hfill\qedsymbol

\begin{remark}
When $d=4$, Lemma \ref{lem:ls_variance} gives a slightly worse
upper bound on \smash{$\sum_{i=1}^n 1/(1+\lambda \rho_i)^2$}, with an
``extra'' term \smash{$(nc_2 / \lambda^{d/2})) \log(1+c_3\lambda)$},
for constants $c_2,c_3>0$. It is not hard to show, by tracing through
the same arguments as given above that we can use this to establish an
upper bound on the max risk of
\begin{equation*}
\sup_{\theta_0 \in \cS_d(\BR_n')}
\E \big[\MSE(\htheta^{\mathrm{LE}}, \theta_0)\big] \leq
\frac{c}{n} \Big( (n\sigma^2)^{\frac{2}{d+2}}
(\BR_n')^{\frac{2d}{d+2}} \log(n/(\BR_n')^2) \wedge n \sigma^2
\wedge n^{2/d}(\BR_n')^2 \Big) + \frac{c\sigma^2}{n},
\end{equation*}
only slightly worse than the minimax optimal rate, by a log factor.

When $d \geq 5$, our analysis provides a much worse bound for the max
risk of Laplacian smoothing, as the integral denoted
$I(d)$ in the proof of Lemma \ref{lem:ls_variance} grows very
large when $d \geq 5$.  We conjecture that this not due to slack
in our proof technique, but rather, to the Laplacian smoothing
estimator itself, since all inequalities the proof can be fairly
tight.
\end{remark}

\section{Utility lemmas used in the proofs of Theorems
  \ref{thm:sobolev_minimax} and \ref{thm:le_ls}}

This section contains some calculations on the partial sums of 
eigenvalues of the Laplacian matrix $L$, for $d$-dimensions
grids. These are useful for the proofs of both Theorem
\ref{thm:sobolev_minimax} and Theorem \ref{thm:le_ls}.

\begin{lemma}
\label{lem:lap_eigenvalues}
Let \smash{$L \in \R^{n\times n}$} denote the graph Laplacian matrix
of a $d$-dimensional grid graph, and $\rho_{i_1,\ldots,i_d}$,
\smash{$(i_1,\ldots,i_d) \in \{1,\ldots,\ell\}^d$} be its
eigenvalues, where
\smash{$\ell=n^{1/d}$}. Then there exists a constant $c>0$ (dependent
on $d$) such that, for any $1 \leq \tau \leq \ell$,
\begin{equation*}
\sum_{(i_1,\ldots,i_d) \in \{1,\ldots,\tau\}^d}
  \rho_{i_1,\cdots,i_d} \leq
c \frac{\tau^{d+2}}{\ell^2}.
\end{equation*}
\end{lemma}

\begin{proof}
The eigenvalues of $L$ can be written explicitly as
\begin{equation}
\label{eq:lap_eigenvalues}
\rho_i = 4 \sin^2\Big(\frac{\pi(i_1-1)}{2\ell}\Big) +
\ldots + 4 \sin^2\Big(\frac{\pi(i_d-1)}{2\ell}\Big), \quad
(i_1,\ldots,i_d) \in \{1,\ldots,\ell\}^d.
\end{equation}
This follows from known facts about the eigenvalues for the Laplacian
matrix of a 1d grid, and the fact that the Laplacin matrix for
higher-dimensional grids can be expressed in terms of a Kronecker sum
of the Laplacian matrix of an appropriate 1d grid (e.g.,
\cite{conte1980,kunsch1994,ng1999,wang2008,graphtf,hutter2016optimal}).
We now use the fact that $\sin(x) \leq x$ for all $x \geq 0$, which
gives us the upper bound
\begin{align*}
\sum_{(i_1,\ldots,i_d) \in \{1,\ldots,\tau\}^d}
\rho_{i_1,\cdots,i_d} & \leq
\frac{\pi^2}{\ell^2} \sum_{(i_1,\ldots,i_d) \in \{1,\ldots,\tau\}^d}
\Big((i_1-1)^2 + \ldots + (i_d-1)^2\Big) \\
&\leq \frac{\pi^2 d}{\ell^2} \tau^{d-1}
\sum_{i=1}^{\tau} (i-1)^2 \\
&\leq \frac{\pi^2 d}{\ell^2} \tau^{d-1} \tau^3 \\
&= \frac{\pi^2 d}{\ell^2} \tau^{d+2},
\end{align*}
as desired.
\end{proof}

\begin{lemma}
\label{lem:ls_variance}
Let \smash{$L \in \R^{n\times n}$} denote the graph Laplacian matrix
of a $d$-dimensional grid graph, and $\rho_i$, $i=1,\ldots,n$ be its
eigenvalues.  Let $\lambda \geq 0$ be arbitrary.
For $d=1,2,$ or $3$, there are constants $c_1,c_2>0$ such that 
\begin{equation*}
\sum_{i=1}^n \frac{1}{(1+\lambda \rho_i)^2}
\leq c_1 + c_2 \frac{n}{\lambda^{d/2}}.
\end{equation*}
For $d=4$, there are constants $c_1,c_2,c_3>0$ such that
\begin{equation*}
\sum_{i=1}^n \frac{1}{(1+\lambda \rho_i)^2}
\leq c_1 + c_2 \frac{n}{\lambda^{d/2}}
\Big( 1 + \log(1+c_3 \lambda) \Big).
\end{equation*}
\end{lemma}

\begin{proof}
We will use the explicit form of the eigenvalues as given in the proof
of Lemma \ref{lem:lap_eigenvalues}.  In the expressions below, we use
$c>0$ to denote a constant whose value may change from line to line.
Using the inequality $\sin x \geq x/2$ for $x \in [0,\pi/2]$,
\begin{align*}
\sum_{i=1}^n \frac{1}{(1+\lambda \rho_i)^2}
&\leq \sum_{(i_1,\ldots,i_d) \in \{1,\ldots,\ell\}^d}
\frac{1}{\big(1 + \lambda \frac{\pi^2}{4\ell^2}
\sum_{j=1}^d  (i_j-1)^2 \big)^2} \\
&\leq 1 + \int_{[0,\ell]^d}
\frac{1}{\big(1 + \lambda \frac{\pi^2}{4}
\sum_{j=1}^d x_j^2 / \ell^2 \big)^2} \, dx \\
&= 1 + c \int_0^{\ell\sqrt{d}}
\frac{1}{\big(1 + \lambda \frac{\pi^2}{4}
r^2 / \ell^2 \big)^2} r^{d-1} \, dr \\
&= 1 + c \frac{n}{\lambda^{d/2}}
\underbrace{\int_0^{\frac{\pi}{2} \sqrt{\lambda d}}
\frac{u^{d-1}}{(1 + u^2)^2} \, du}_{I(d)}.
\end{align*}
In the second inequality, we used the fact that the right-endpoint
Riemann sum is always an underestimate for the integral of a function
that is monotone nonincreasing in each coordinate.  In the third, we
made a change to spherical coordinates, and suppressed all of the
angular variables, as they contribute at most a constant factor.  It
remains to compute $I(d)$, which can be done by symbolic integration:
\begin{align*}
I(1) &= \frac{\pi \sqrt{d}}{4\big(1+\frac{\pi^2}{4}\lambda d\big)} + 
\frac{1}{2} \tan^{-1}
\Big(\frac{\pi}{2} \sqrt{\lambda d}\Big) \leq \frac{1}{4} + \frac{\pi}{4}, \\
I(2) &= \frac{1}{2} - \frac{1}{2\big(1+\frac{\pi^2}{4}\lambda d\big)}
\leq \frac{1}{2} ,\\
I(3) &= \frac{1}{2} \tan^{-1}
\Big(\frac{\pi}{2} \sqrt{\lambda d}\Big) \leq \frac{\pi}{4}, 
\quad \text{and} \\
I(4) &= \frac{1}{2} \log \Big(1+ \frac{\pi^2}{4} \lambda d \Big) +
\frac{1}{2\big(1+ \frac{\pi^2}{4} \lambda d\big)} - \frac{1}{2} \leq
\frac{1}{2} \log \Big(1+ \frac{\pi^2}{4} \lambda d \Big) + \frac{1}{2}.
\end{align*}
This completes the proof.
\end{proof}

\section{Additional experiments comparing TV denoising and Laplacian
  smoothing for piecewise constant functions} 

\begin{figure}[H]
  \centering
  \begin{tabular}{cc}
    Piecewise constant signal in 2d &
    Piecewise constant signal in 3d \\
    {\includegraphics[width=0.45\textwidth]{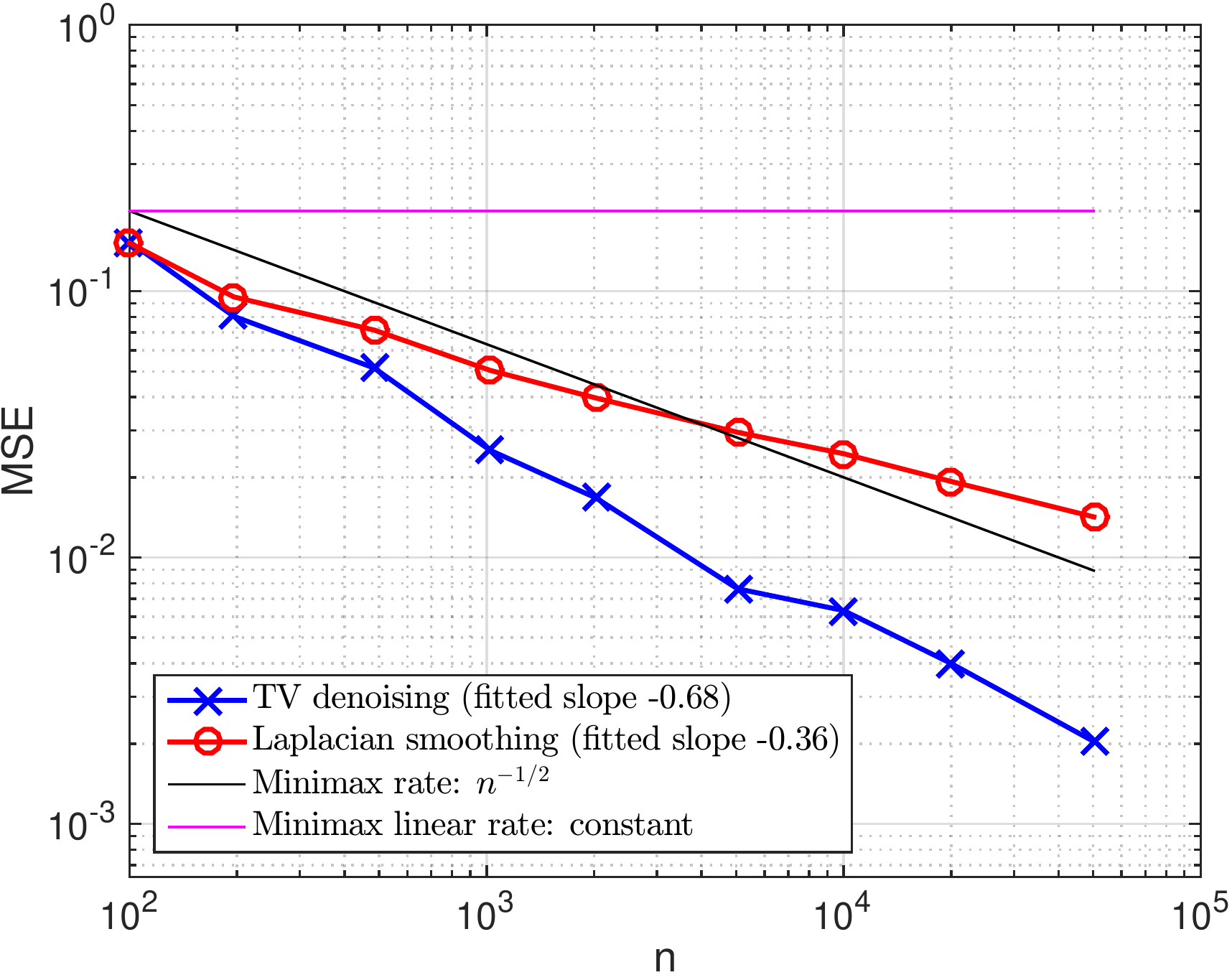}} &
    {\includegraphics[width=0.45\textwidth]{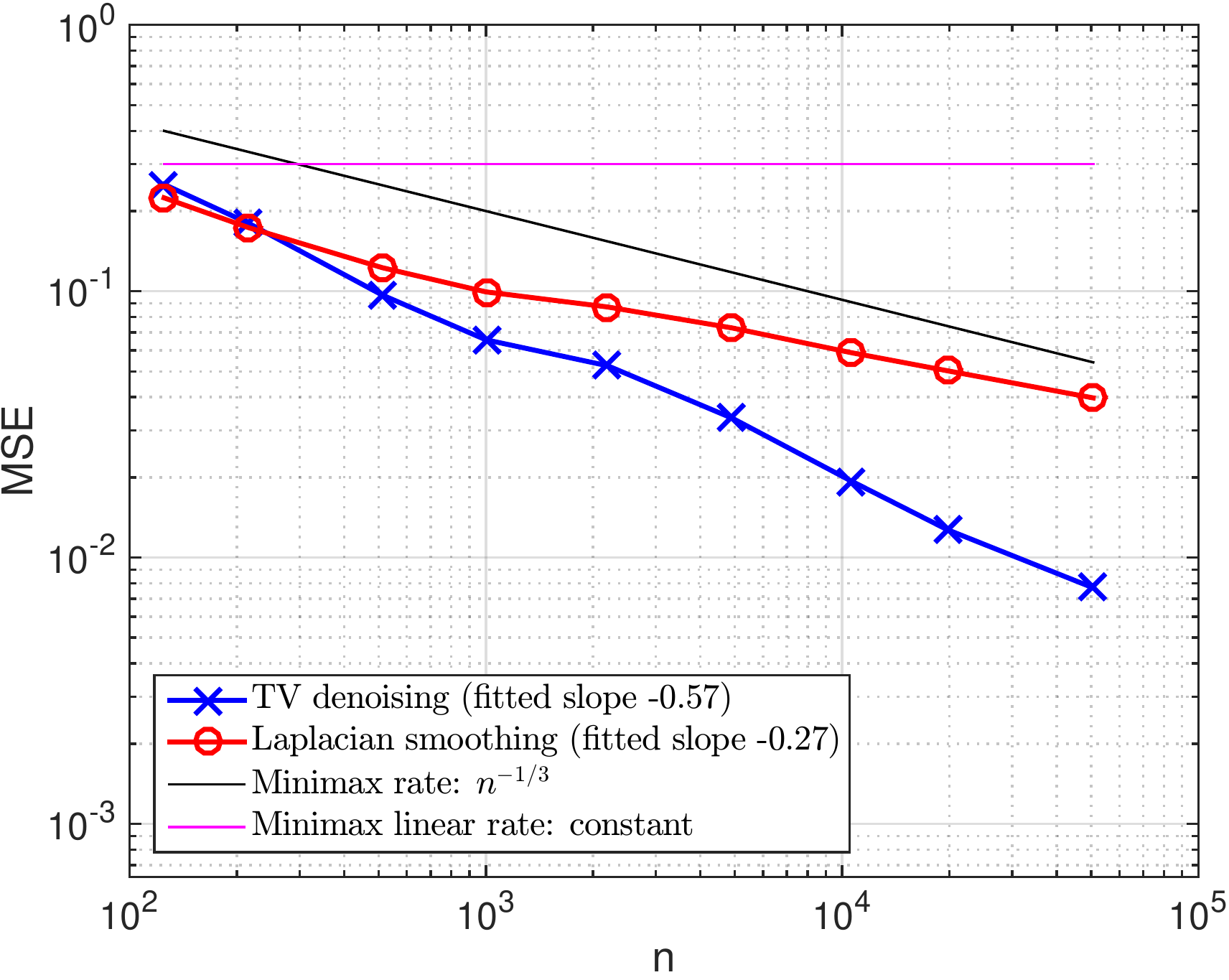}}
  \end{tabular}
\vspace{-5pt}
\caption{\small\it MSE curves for estimating a ``piecewise constant''
  signal, having a single elevated region, over 2d and 3d grids.  For 
  each $n$, the results were averaged over 5 repetitions, and the
  Laplacian smoothing and TV denoising estimators were tuned for best
  average MSE performance.  We set $\theta_0$ to satisfy
  \smash{$\|D\theta_0\|_1 \asymp n^{1-1/d}$}, matching the canonical
  scaling.  Note that all estimators achieve better performance than
  that dictated by their minimax rates.}   
\label{fig:piecewise_const} 
\end{figure}

\bibliographystyle{plainnat}
\bibliography{ryantibs}

\end{document}